%% file: chaumontfrelet_ingremeau_2023a.tex
\documentclass[a4paper,11pt]{amsart}

\usepackage[foot]{amsaddr}

\usepackage[breaklinks,bookmarks=false]{hyperref}
\hypersetup{colorlinks, linkcolor=blue, citecolor=blue,urlcolor=blue,plainpages=false}

\usepackage{a4wide}
\usepackage{amssymb}

\newcommand{\HH}{\widehat{H}}

\input{general_commands}
\input{special_commands}

\newtheorem{theorem}{Theorem}[section]
\newtheorem{lemma}[theorem]{Lemma}
\newtheorem{corollary}[theorem]{Corollary}
\newtheorem{proposition}[theorem]{Proposition}

\numberwithin{equation}{section}

\usepackage{xcolor}

\title{Decay of coefficients and approximation rates in Gabor Gaussian frames}
\author{T. Chaumont-Frelet$^\star$}
\author{M. Ingremeau$^\dagger$}

\address{\vspace{-.5cm}}
\address{\noindent \tiny \textup{$^\star$University C\^ote d'Azur, Inria, CNRS, LJAD}}
\address{\noindent \tiny \textup{$^\dagger$University C\^ote d'Azur, CNRS, LJAD}}

\begin{document}

\maketitle
\thispagestyle{empty}

\begin{abstract}
The aim of this note is to present a self-contained proof of the fact that a function
can be approximated using a linear combination of Gaussian coherent states, with a number
of terms controlled in terms of the smoothness and of the decay at infinity of the function.
This result, which is essential in \cite{chaumontfrelet_dolean_ingremeau_2022a}, can easily
be obtained using advanced results on modulation spaces, but the proof presented here is
completely elementary and self-contained.

\vspace{.5cm}
\end{abstract}

\section{Introduction}

In \cite{chaumontfrelet_dolean_ingremeau_2022a}, the authors proposed a novel approach to solve numerically a family of PDEs with a small parameter,  including the Helmholtz equation at large frequencies. This approach relied on some ideas coming from semiclassical analysis, as well as on results concerning \emph{Gabor frames}.

Gabor frames are a standard tool to decompose functions into
a discrete sum of ``coherent states'', which are
localised both in position and Fourier spaces. Such expansions are somehow
similar to Fourier expansions, but are more subtle, as Gabor frames do not
form orthonormal bases.  
The study of Gabor frames
is a field of research which has been very active for the past decades;  we refer the reader to the monographs  \cite{grochenig2001foundations, christensen2003introduction,  benyi2020modulation} for an account of recent advances.

The main result of time-frequency analysis which is used in \cite{chaumontfrelet_dolean_ingremeau_2022a} is  a decay property for the coefficients of functions when decomposed in a Gabor frame, in terms of the regularity of
the functions and of their decay at infinity. These results are analogous to the
standard decay properties of Fourier coefficients, and permit to show that
a finite number of coherent states provide a good approximation to any smooth rapidly decaying function.

Such a result can be easily deduced from advanced results of time-frequency analysis concerning  \emph{modulation spaces}., as we will explain in Appendix \ref{Sec:Modulation}.  

The aim of this note is to give a completely self-contained proof of all the results concerning Gabor frames which are used in \cite{chaumontfrelet_dolean_ingremeau_2022a}, and in particular, of the fact that a smooth rapidly decaying function can be approximated using few coherent states.  This note is thus intended mainly for readers who want to understand where the results of \cite{chaumontfrelet_dolean_ingremeau_2022a} come from,  without learning the recent developments in time-frequency analysis.

We insist on the fact that the results presented here are not new,  and are not as general as they could be; for instance,  we will consider only Gabor frames built from Gaussian windows (which simplifies the proofs), while more general window functions could be considered using modulation spaces.

The remainder of the document is organized as follows.
In Section \ref{section_settings}, we introduce the main notations we will use in this
work. In Section \ref{Sec:Frame}, we recall standard properties of
Gabor frames and state the main results that are used in \cite{chaumontfrelet_dolean_ingremeau_2022a}. In Section \ref{sec:proofframe}, we recall the proof of the frame property for the Gabor system.  In Section \ref{sec:Decay}, we prove properties
of decay of the coefficients, which will be essential in the proof of our main result, given in
Section \ref{Sec:Proof}. Finally, Appendices \ref{App:Elem}, \ref{App:DecayPSStar} and
\ref{appendix_sharp} collect necessary technical results, while Appendix \ref{Sec:Modulation} explains how the main result of this note can be recovered using results concerning modulation spaces.

\section{Settings}
\label{section_settings}

\subsection{Scaling parameter}

In the sequel, all our quantities will depend on a parameter $k \geq 1$,
which is an arbitrary but fixed real number. This is more convenient
for applications to PDEs, where it is useful to choose $k$ related to
the wavelength of the solution.

Although we may not indicate it explicitly, all our results hold true
for any value of $k$, with the estimates holding uniformly. Besides,
$d \in \N^\star$ denotes the number of space dimensions. 

\subsection{Multi-indices}
For a multi-index $\ba \in \N^d$, $[\ba] \eq a_1+\dots+a_d$
denotes its usual $\ell^1$ norm. If $v: \R^d \to \C$, the notation
\begin{equation*}
\partial^{\ba} v
\eq
\frac{\partial^{a_1}}{\partial x_1}
\dots
\frac{\partial^{a_d}}{\partial x_d} v
\end{equation*}
is employed for the partial derivatives in the sense of distributions, whereas
$\bx^{\ba} \eq x_1^{a_1} \cdot \ldots \cdot x_d^{a_d}$.

\subsection{Functional spaces}

$L^2(\R^d)$ is the usual space of complex-valued square integrable
functions over $\R^d$. We respectively denote by $\|\cdot\|_{L^2(\R^d)}$
and $(\cdot,\cdot)$ its norm and inner-product. If $\omega \subset \R^d$ is a measurable
subset and $v \in L^2(\R^d)$, we set $\|v\|_{L^2(\omega)} = \|\chi_\omega v\|_{L^2(\R^d)}$,
where $\chi_\omega$ is the set function of $\omega$. We also denote by
$H^p(\R^d)$ the usual Sobolev space of order $p \in \N$, that we equip with
the norm
\begin{equation*}
\|v\|_{H^p_k(\R^d)}^2
\eq
\sum_{[\ba] \leq p} k^{-2[\ba]} \|\partial^{\ba} v\|_{L^2(\R^d)}^2
\quad
\forall v \in H^p(\R^d).
\end{equation*}
Notice that the above norm is equivalent to the usual Sobolev norm,
but with equivalence constant depending on $k$. We refer the reader
to, e.g., \cite{adams_fournier_2003a} for an in-depth presentation
of the above spaces.

We will also need a family of weighted Sobolev spaces. We first introduce the norms
\begin{equation*}
\|v\|_{\HH^p_k(\R^d)}^2
\eq
\sum_{[\ba] \leq p}
\sum_{q=0}^{p-|\ba|}
k^{-2[\ba]} \||\bx|^{q} \partial^{\ba} v\|_{L^2(\R^d)}^2
\qquad \forall v \in \CS(\R^d)
\end{equation*}
for all $p \in \N$, and we define $\HH^p(\R^d)$ as the closure 
of $\CS(\R^d)$ in $L^2(\R^d)$ with respect to $\|{\cdot}\|_{\HH^p_k(\R^d)}$.
For future references, we note that if $u \in \HH^{p+1}(\R^d)$,
we have $x_j u \in \HH^p(\R^d)$ with 
\begin{equation}
\label{eq_norm_xu}
\|x_j u\|_{\HH^{p}_k(\R^d)} \leq C_p  \|u\|_{\HH^{p+1}_k(\R^d)}
\end{equation}
for $1 \leq j \leq d$. 

\subsection{Fourier transform}

For $v \in L^2(\R^d)$, we define a weighted Fourier transform by
\begin{equation*}
\CF_k(v)(\bxi)
\eq
\left(\frac{k}{2\pi}\right)^{d/2} \int_{\R^d} v(x) e^{-ik\bx \cdot \bxi} \dx
\end{equation*}
for a.e. $\bxi \in \R^d$. For all $v \in L^2(\R^d)$, we then have
\begin{equation*}
\|\CF_k(v)\|_{L^2(\R^d)} =  \|v\|_{L^2(\R^d)}.
\end{equation*}
In addition, if $v \in H^1(\R^d)$ or respectively if $|\bx| v \in L^2(\R^d)$, then
\begin{equation*}
\CF_k(\partial_j v) = ik \xi_j \CF_k(v),
\qquad
ik \CF_k(x_j v) = \partial_j \CF_k(v)
\end{equation*}
for $j \in \{1,\dots,d\}$. In particular, for any $p\in \N$, we have for every $v\in H^p(\R^d)$
\begin{equation}\label{eq:SobNormFourier}
\|v\|_{H^p_k(\R^d)}^2 = \sum_{[\ba] \leq p}  \|\bxi^{\ba} \CF_k v\|_{L^2(\R^d)}^2.
\end{equation}

\subsection{Generic constants}

Throughout the manuscript, $C$ denotes a constant that may vary from one occurrence
to the other and only depends on the space dimension $d$. If $p,q,\dots$ are previously
introduced symbols, then the notation $C_{p,q,\dots}$ is similarly used to indicate a
constant that solely depends on $d$ and $p,q,\dots$.

\section{Main results concerning Gabor frames of Gaussian states}
\label{Sec:Frame}

\subsection{The Gabor frame}

For $[\bm,\bn] \in \Z^{2d}$, we set
\begin{equation}
\label{eq_def_xx_xi}
\xxx{k}{\bm}
\eq
\sqrt{\frac{\pi}{k}}\bm,
\qquad
\xxi{k}{\bn}
\eq
\sqrt{\frac{\pi}{k}}\bn,
\end{equation}
so that, the couples $[\xxkm,\xikn]$ form a lattice of the phase space $\Z^{2d}$.
We associate with each point $[\bm,\bn]$ in the lattice the Gaussian state
\begin{equation*}
\gs{k}{\bm}{\bn} (\bx)
\eq
\left (\frac{k}{\pi} \right )^{d/4}
e^{-\frac{k}{2}|\bx-\xxx{k}{\bm}|^2}
e^{ik(\bx-\xxx{k}{\bm}) \cdot \xxi{k}{\bn}}.
\end{equation*}

The set $(\gskmn)_{[\bm,\bn] \in \Z^{2d}}$ forms a \emph{frame}, i.e.,
there exist two constants $\alpha,\beta > 0$ that only depends on $d$, such that
\begin{equation}\label{eq_frame}
\alpha^2 \|u\|^2_{L^2(\R^d)}
\leq
\sum_{[\bm,\bn] \in \Z^{2d}}|(u,\gskmn)|^2
\leq
\beta^2 \|u\|^2_{L^2(\R^d)}
\quad
\forall u \in L^2(\R^d).
\end{equation}
This fact was first shown in \cite{daubechies_grossman_meyer_1986a} when $d=1$
and $k=1$ (from which the case $d>1$ and $k\neq 1$ follow easily), and we will recall their proof in Section \ref{sec:proofframe}. We refer the reader to \cite{christensen2003introduction} for generalisations of this result.

\subsection{The dual frame}

Given $u \in L^2(\R^d)$, we call the sequence $((u,\gskmn))_{[\bm,\bn] \in \Z^{2d}}$
the coefficients of $u$ in the frame, and we introduce the operator
$T_k: L^2(\R^d) \to \ell^2(\Z^{2d})$ mapping a function $u$ to its coefficients:
\begin{equation}\label{eq:DefCoefOp}
(T_k u)_{\bm,\bn} \eq (u,\gskmn) \qquad [\bm,\bn] \in \Z^{2d}.
\end{equation}
The fact that $T_k$ maps into $\ell^2(\Z^{2d})$ is a direct consequence of \eqref{eq_frame}.
Straightfoward computations show that the adjoint $T_k^\star: \ell^2(\Z^{2d}) \to L^2(\R^d)$
of $T_k$ is given by
\begin{equation*}
T_k^\star U = \sum_{[\bm,\bn] \in \Z^{2d}} U_{\bm,\bn} \gskmn,
\end{equation*}
for all $U \in \ell^2(\Z^{2d})$ and, in particular, we have
\begin{equation}
\label{eq_TksTku}
(T_k^\star T_k) u = \sum_{[\bm,\bn] \in \Z^{2d}} (u,\gskmn) \gskmn \quad \forall u \in L^2(\R^d).
\end{equation}
The operator $(T_k^\star T_k)$ is self-adjoint, and the estimates in \eqref{eq_frame}
imply that it is bounded from above and from below. As a result,
$(T_k^\star T_k)$ is invertible, and we will we denote by $S_k$ its inverse.

We call set of functions $(\gskmn^\star)_{\bm,\bn \in \Z^d}$ defined by
\begin{equation*}
\gskmn^\star \eq S_k \gskmn \quad \forall [\bm,\bn] \in \Z^{2d},
\end{equation*}
the ``dual frame'' of $(\gskmn)_{[\bm,\bn] \in \Z^{2d}}$. It is indeed a frame,
since, for any $u\in L^2(\R^d)$, we have
\begin{equation*}
\sum_{[\bm,\bn]\in \Z^{2d}} |(u,\gskmn^\star)|^2
=
\sum_{[\bm,\bn]\in \Z^{2d}} |(u,S_k \gskmn)|^2
=
\sum_{[\bm,\bn]\in \Z^{2d}} |(S_k u,\gskmn)|^2
=
\|T_k S_k u\|^2,
\end{equation*}
and it suffices to use \eqref{eq_frame} and to note that $S_k$ is continuous
with a continuous inverse.

Importantly, recalling \eqref{eq_TksTku}, we see that for all $u \in L^2(\R^d)$, we have
\begin{equation}
\label{eq:ResolutionOfIdentity}
u
=
\sum_{[\bm,\bn] \in \Z^{2d}} (u,\gskmn^\star)\gskmn
\end{equation}
and
\begin{equation}
\label{eq:OtherResolutionIdentity}
u
=
\sum_{[\bm,\bn] \in \Z^{2d}} (u,\gskmn)\gskmn^\star.
\end{equation}

\subsection{Approximation in Gabor frames}

The main result of this note is a generalization of \eqref{eq:ResolutionOfIdentity}
when the series is truncated to a finite number of terms. Specifically, considering
$D > 0$ such that $k^{1/2}D > \sqrt{\pi}$, we define an ``interpolation operator''
$\Pi_D: L^2(\R^d) \to L^2(\R^d)$ by
\begin{equation}
\label{eq_definition_pi}
\Pi_D u \eq \sum_{|(\xxkm,\xikn)| \leq D} (u,\gskmn^\star) \gskmn.
\end{equation}
Notice that, recalling the definitions of $\xxkm$ and $\xikn$ in
\eqref{eq_def_xx_xi}, the requirement that $k^{1/2}D > \sqrt{\pi}$ simply
means that there is at least one point in the summation set, so that it is
not a restrictive assumption.

\begin{theorem}[Approximation in the Gabor frame]
\label{theorem_approximation}
Consider two non-negative integers $p$ and $r$. For all $u \in \HH^{p+r}(\R^d)$,
we have $\Pi_D u \in \HH^p(\R^d)$. In addition, the following estimate holds true:
\begin{equation}
\label{eq_approximation}
\|u - \Pi_D u\|_{\HH^p_k(\R^d)}
\leq
C_{p,r}  D^{-r} \|u\|_{\HH^{p+r}_k(\R^d)}.
\end{equation}
\end{theorem}

The estimate in Theorem \ref{theorem_approximation} is sharp.
Specificially, we have the following result.

\begin{proposition}[Sharp estimate]
\label{proposition_sharp}
For all $p,r \in \N$, there exists a constant $D_\star > 0$ depending on $p$, $r$
and $d$, such that for all $D \geq D_\star$, there exists $u \in \HH^{p+r}(\R^d)$
with
\begin{equation}
\label{eq:SharpUp}
\|u - \Pi_D u\|_{\HH^p_k(\R^d)} \geq C_{p,r} D^{-r} \|u\|_{\HH^{p+r}(\R^d)}.
\end{equation}
\end{proposition}

Theorem \ref{theorem_approximation} will be proved in Sections \ref{sec:Decay} and \ref{Sec:Proof}, and we will give an alternative proof of it in Appendix \ref{Sec:Modulation}. Proposition \ref{proposition_sharp} will be proven in Appendix \ref{appendix_sharp}.

\section{Proof of the frame property}\label{sec:proofframe}

The aim of this section is to recall the proof of (\ref{eq_frame}),  following closely \cite{daubechies_grossman_meyer_1986a}.

First of all, note that for $f \in L^2(\R^d)$, if we set 
\begin{equation}\label{eq:Rescaling}
(\delta_k f) (x) := k^{-d/4} f\left(k^{-1/2} x \right),
\end{equation}
we have $\|\delta_k f\|_{L^2(\R^d)} = \|f\|_{L^2(\R^d)}$ and
\begin{equation}\label{eq:DefRescaling}
\delta_k \gskmn = \Psi_{1, \bm, \bn}  .
\end{equation}
Therefore, it suffices to prove \eqref{eq_frame} for $k=1$, and the general case will follow.
To lighten notations, we shall write $\Psi_{\bm,\bn}$ instead of $\Psi_{1,\bm,\bn}$ in this section.

Letting $\square \eq [-\sqrt{\pi}, \sqrt{\pi} )^d \times [ - \sqrt{\pi}/2, \sqrt{\pi}/2)^d$,
we define the Zak transform $Z : L^2(\R^d) \longrightarrow L^2(\square)$ as follows.
If $f\in C_c(\R^d)$, we set
\begin{equation*}
(Z f)(\bv,\bq)
\eq
\frac{1}{(2\sqrt{\pi})^{d/2}}
\sum_{\bell\in \Z^d} e^{i \sqrt{\pi} \bv\cdot \bell} f(\bq- \sqrt{\pi}\bell)
\qquad
\forall (\bv,\bq) \in \square.
\end{equation*}
One readily checks that $\|Z f\|_{L^2(\R^d)}= \|f\|_{L^2(\square)}$, so that $Z$ extends
to a unitary transform.

Let $\bm,\bn\in \Z^d$. We have
\begin{align*}
(Z \Psi_{\bm,\bn})(\bv,\bq)
&=
\frac{1}{2^{d/2} \pi^{3d/4}} \sum_{\bell\in \Z^d}
e^{i \sqrt{\pi} \bv\cdot \bell}
e^{i (\bq- \sqrt{\pi}\bell-\sqrt{\pi}\bm)\cdot \sqrt{\pi}\bn}
e^{-\frac{|\bq- \sqrt{\pi}\bell-\sqrt{\pi}\bm |^2}{2}}
\\
&=
\frac{1}{2^{d/2} \pi^{3d/4}}
e^{i \sqrt{\pi} \bv\cdot \bm} e^{i \bq\cdot \sqrt{\pi}\bn}
\sum_{\bell' \in \Z^d}
e^{i \sqrt{\pi} \bv\cdot \bell'}
e^{i \pi \bell' \cdot \bn}
e^{-\frac{|\bq- \sqrt{\pi}\bell'|^2}{2}}
\\
&=
e^{i \sqrt{\pi} \bv\cdot \bm} e^{i \bq\cdot \sqrt{\pi}\bn} (Z \Psi_{0,\bn'}) (\bv,\bq),
\end{align*}

The, given $f \in L^2(\R^d)$, we can write
\begin{align*}
\|T_1 f\|_{\ell^2}^2
&=
\sum_{\bn'\in \{0,1\}^d} \sum_{\bm, \bp \in \Z^d}
\left|\left\langle Z \Psi_{\bm,2\bp+\bn'}, Z f \right\rangle \right|^2
\\
&=
\sum_{\bn'\in \{0,1\}^d} \sum_{\bm, \bp \in \Z^d}
\left | \int_{\square}
(Z f)(\bv,\bp)  e^{-i \sqrt{\pi} \bv\cdot \bm} e^{-i \bq\cdot \sqrt{\pi}(2\bp+\bn')}
(Z \Psi_{0,\bn'})(\bv,\bq)\mathrm{d}\bv \mathrm{d}\bq
\right |^2
\\
&=
(2\pi)^d \int_{\square} |(Z f)(\bv,\bq)|^2
\sum_{\bn'\in \{0,1\}^d} \left| Z \Psi_{0, \bn'}(\bv,\bq) \right|^2
\mathrm{d}\bv \mathrm{d}\bq,
\end{align*}
by Plancherel's formula. Therefore, if we introduce the function 
\begin{equation}\label{eq:Theta}
\Theta(\bv,\bq) \eq \sum_{\bn'\in \{0,1\}^d} \left| Z \Psi_{0, \bn'}(\bv,\bq) \right|^2
\qquad
\forall (\bv,\bq) \in \square,
\end{equation}
we have
\begin{equation*}
(2\pi)^d \min_{\bv,\bq} \Theta(\bv,\bq) \|f\|^2_{L^2(\R^d)}
\leq \|T_1 f \|_{\ell^2}^2
\leq
(2\pi)^d \max_{\bv,\bq} \Theta(\bv,\bq) \|f\|^2_{L^2(\R^d)}.
\end{equation*}

We define $\theta$ to be $\Theta$ when $d=1$. That is, for
$(v,q) \in  \left[-\sqrt{\pi}, \sqrt{\pi} \right) \times \left[ -\sqrt{\pi}/2, \sqrt{\pi}/2 \right)$,
we set
\begin{equation*}
\theta(v,q)
\eq
\frac{1}{2\pi} \left|
\sum_{\ell \in \Z}
e^{i \sqrt{\pi} v\ell }
e^{-(q- \sqrt{\pi} \ell)^2)/2}
\right |^2
+
\frac{1}{2\pi} \left |
\sum_{\ell \in \Z}
e^{i \sqrt{\pi} v \ell}
e^{i\sqrt{\pi} (q-\sqrt{\pi} \ell)}
e^{-(q - \sqrt{\pi}\ell)^2)/2}
\right |^2.
\end{equation*}
We then have, for $\bv= (v_1,..., v_d)$, $\bq= (q_1,..., q_d)$
\begin{equation*}
\Theta(\bv,\bq) = \prod_{j=1}^d \theta(v_j, q_j),
\end{equation*}
and therefore
\begin{equation*}
\min\limits_{\bv,\bq} \Theta(\bv,\bq)= \left(\min\limits_{v,q}  \theta(v,q)\right)^d
\qquad
\max\limits_{\bv,\bq} \Theta(\bv,\bq)= \left(\max\limits_{v,q} \theta(v,q)\right)^d.
\end{equation*}
Recalling that the Jacobi Theta functions are given, for any $z, \tau \in \C$ with $\Im \tau >0$
by
$$\vartheta(z, \tau) := \sum_{n\in \Z} \exp (\pi i n^2 \tau + 2i\pi n z),$$
we see that
\begin{align*}\theta(v,q) &= \frac{e^{-q^2}}{2\pi} \left|\sum_{\ell \in \Z} \exp \left( -\frac{\pi \ell^2}{2}  + \ell (i v  \sqrt{\pi} + q \sqrt{\pi})  \right) \right|^2 +  \frac{e^{-q^2}}{2\pi}  \left|\sum_{\ell \in \Z} \exp \left( -\frac{\pi \ell^2}{2}  + \ell (i v  \sqrt{\pi} + q \sqrt{\pi}) - i\pi  \right) \right|^2  \\
&= \frac{e^{-q^2}}{2\pi} \left( \left|\vartheta \left(\frac{v- i q }{2  \sqrt{\pi}}, \frac{i}{2} \right)\right|^2 + \left|\vartheta \left( \frac{v-iq}{2 \sqrt{\pi}} - \frac{1}{2}, \frac{i}{2} \right)\right|^2 \right).
\end{align*}

This gives us the upper bound in (\ref{eq_frame}), buy continuity of $\vartheta$. To obtain the lower bound, we note that  the quantity $\vartheta(z,\tau)$ vanishes if and only if we have
$$z= \left(n+ \frac{1}{2}\right) + \tau \left(m+ \frac{1}{2}\right)$$
for $m, n\in \Z$.
Hence, the term $\left|\vartheta \left(\frac{v- i q}{2  \sqrt{\pi}}, \frac{i}{2} \right)\right|^2$ vanishes if and only if $v=- \sqrt{\pi}$ and $q= - \frac{\sqrt{\pi}}{2}$. On the other hand, the term $ \left|\vartheta \left( \frac{v-iq}{2 \sqrt{\pi}} - \frac{1}{2}, \frac{i}{2} \right)\right|^2$ vanishes if and only if $v=0$ and $q= - \frac{\sqrt{\pi}}{2}$.
Therefore, $\min_{v,q} \theta(v, q)>0$, which proves the lower bound in \eqref{eq_frame}.

\section{Decay of coefficients in the frames}\label{sec:Decay}

\subsection{Decay of coefficients in the dual frame}

Recall that the estimates in \eqref{eq_frame}
imply that whenever $u \in L^2(\R^d)$, its coefficients belongs to $\ell^2$,
with an $\ell^2$ norm comparable to $\|u\|_{L^2(\R^d)}$. 
In Proposition
\ref{proposition_decay} we generalize this idea by showing that if $u$
exhibits more regularity, stronger decay properties hold true for its
coefficients sequence.

\begin{proposition}[Coefficients' decay]
\label{proposition_decay}
For all $p \in \N$ and $u \in \HH^p(\R^d)$, we have
\begin{equation}
\label{eq_decay}
\sum_{[\bm, \bn]\in \Z^{2d}}
\left (|\xxkm|^2 + |\xikn|^2\right )^p
|(u,\gskmn)|^2
\leq
C_p  \|u\|^2_{\HH^p_k(\R^d)}.
\end{equation}
\end{proposition}

\begin{proof}
We will prove \eqref{eq_decay} by induction. Note that, for $p=0$,
it is an immediate consequence of \eqref{eq_frame}. Then, suppose that there exists
$p \in \N$ such that \eqref{eq_decay} holds for all $q\leq p-1$, and let
us show that it also holds for $q=p$.

Let us fix $j\in \{1,\dots,d\}$. The polynomials $(Q_\ell)_{\ell=0}^p$ from Lemma
\ref{lemma_differentiation} all have different degrees, so that they form a basis
of $\R_{p}[X]$. In particular, we may find real numbers $(a_{\ell})_{\ell=0}^p$
solely depending on $p$ such that
\begin{equation*}
z^p = \sum_{\ell=0}^p a_\ell Q_\ell(z),
\end{equation*}
and in particular, introducing for the sake of shortness the function
$y: \bx \to (\xxkm-\bx+i\xikn)_j$, we have
\begin{equation}
\label{tmp_poly_expr}
k^{p/2} y^p = \sum_{\ell=0}^p
a_{\ell} Q_\ell(k^{1/2} y).
\end{equation}
Then, using \eqref{tmp_poly_expr}, and recalling \eqref{eq_diff_gskmn}, we obtain
\begin{equation*}
k^{p/2}
(\overline{y}^p u, \gskmn)
=
\sum_{\ell=0}^p a_\ell (u, Q_\ell(k^{1/2}y)\gskmn)
=
\sum_{\ell=0}^pa_\ell k^{-\ell/2}
\left(u, \partial_j^\ell \gskmn \right).
\end{equation*}
and integration by parts shows that
\begin{equation*}
(\overline{y}^p u, \gskmn)
=
\sum_{\ell=0}^p (-1)^\ell a_\ell k^{-(p+\ell)/2} (\partial_j^\ell u,\gskmn),
\end{equation*}
where we also multiplied both sides by $k^{-p/2}$.
Recalling the definition of $y$ and using the upper-bound in \eqref{eq_frame},
after summing over $[\bm,\bn] \in \Z^{2d}$, we see that
\begin{multline}
\label{tmp_first_term}
\sum_{[\bm, \bn]\in \Z^{2d}}
|((x_j^{k,\bm}-i\xi_j^{k,\bn}+x_j)^p u,\gskmn)|^2
\\
\leq
C_p \sum_{\ell=0}^p
a_\ell^2 k^{-(\ell+p)}
\|\partial_j^\ell u\|_{L^2(\R^d)}^2
\leq
C_p \sum_{\ell=0}^p
a_\ell^2 k^{-2\ell}
\|\partial_j^\ell u\|_{L^2(\R^d)}^2
\leq
C_p \|u\|_{H_k^p(\R^d)}^2,
\end{multline}
where we have used the facts that the numbers $a_\ell$ only depend on $p$
and that $k \geq 1$. Since we have 
\begin{equation*}
\left (x_j^{k, \bm} - i\xi_j^{k,\bn}\right )^p
=
\left (x_j^{k, \bm} - x_j - i\xi_j^{k,\bn}\right )^p
-
\sum_{q=0}^{p-1} \binom{p}{q} (-x_j)^{p-q} \left (x_j^{k, \bm} - i\xi_j^{k,\bn}\right )^q,
\end{equation*}
and
\begin{equation*}
\left (
(x_j^{k,\bm})^2 + (\xi_j^{k,\bn})^2
\right )^p
|(u,\gskmn)|^2 
=
\left |
\left (
((x_j^{k,\bm} - i\xi_j^{k,\bn})^p u,\gskmn
\right )
\right |^2,
\end{equation*}
we may write that
\begin{multline}
\label{tmp_result_j}
\left (
(x_j^{k,\bm})^2 + (\xi_j^{k,\bn})^2
\right )^p
|(u,\gskmn)|^2 
\leq
\\
C_p
\left (
\left |
\left (
((x_j^{k,\bm} - x_j - i\xi_j^{k,\bn})^p u,\gskmn
\right )
\right |^2
+
\sum_{q=0}^{p-1}\left ((x_j^{k,\bm})^2+(\xi_j^{k,\bn})^2\right )^q
|(x_j^{p-q} u,\gskmn)|^2
\right ).
\end{multline}
Thanks to the induction hypothesis, we have
\begin{equation*}
\sum_{q=0}^{p-1}
\sum_{[\bm,\bn] \in \Z^{2d}}
\left ((x_j^{k,\bm})^2+(\xi_j^{k,\bn})^2\right )^q
|(x_j^{p-q} u,\gskmn)|^2
\leq
C_p
\sum_{q=0}^{p-1}
\|x^{p-q}u\|^2_{\HH^q_k(\R^d)},
\end{equation*}
and \eqref{eq_norm_xu} then ensures that
\begin{equation}
\label{tmp_second_term}
\sum_{q=0}^{p-1}
\sum_{[\bm,\bn] \in \Z^{2d}}
\left ((x_j^{k,\bm})^2+(\xi_j^{k,\bn})^2\right )^q
|(x_j^{p-q} u,\gskmn)|^2
\leq
C_p  \|u\|_{\HH^p_k(\R^d)}^2.
\end{equation}
We then sum \eqref{tmp_result_j} over $[\bm,\bn] \in \Z^{2d}$. Using
\eqref{tmp_first_term} and \eqref{tmp_second_term} to respectively
bound the first and second term in the right-hand side, we obtain
\begin{equation*}
\sum_{[\bm,\bn] \in \Z^{2d}}
\left (
(x_j^{k,\bm})^2 + (\xi_j^{k,\bn})^2
\right )^p
|(u,\gskmn)|^2 
\leq
C_p \|u\|_{\HH^p_k(\R^d)}^2,
\end{equation*}
and \eqref{eq_decay} follows by summing over $j \in \{1,\dots,d\}$,
which concludes the proof by induction.
\end{proof}

\subsection{Decay of coefficients in the initial frame}

We now prove an analogue of Proposition \ref{proposition_decay} for the
coefficients of the dual coefficients $(u,\gskmn^\star)$, which are used
for the expansion in the initial frame $(\gskmn)_{[\bm,\bn] \in \Z^{2d}}$.

We start with a preliminary result concerning the inner product of elements
of the dual frame. Identity \eqref{eq:DeuxiemeDecompoPsiStar} follows directly from
\eqref{eq:ResolutionOfIdentity}, while \eqref{eq:PSStar} can be deduced from
\cite[Corollary 3.7]{fornasier_grochenig_2005a}. For the reader's convenience,
we give a self-contained elementary proof of \eqref{eq:PSStar} in Appendix
\ref{App:DecayPSStar}.

\begin{proposition}[Expansion of the dual frame]
\label{proposition_decomposition_gskmn_star}
For all $[\bm,\bn] \in \Z^{2d}$, we have
\begin{equation}
\label{eq:DeuxiemeDecompoPsiStar}
\gskmn^\star
=
\sum_{[\bm',\bn']\in \Z^{2d}} (\gskmn^\star,\gskmnp^\star) \gskmnp.
\end{equation}
In addition, for all $\varepsilon > 0$, there exists a constant $C_\varepsilon > 0$
such that
\begin{equation}
\label{eq:PSStar}
\left | \left( \gskmn^\star, \gskmnp^\star \right ) \right|
\leq
C_\varepsilon
e^{-| [\bm, \bn] - [\bm', \bn'] |^{1-\varepsilon}}.
\end{equation}
\end{proposition}

We are now ready to establish our result concerning the coefficients' decay.

\begin{proposition}[Dual coefficients' decay]
\label{proposition_decay_star}
For all $p \in \N$ and $u \in \HH^p(\R^d)$, we have
\begin{equation}
\label{eq_decay_star}
\sum_{[\bm, \bn]\in \Z^{2d}}
\left (|\xxkm|^2 + |\xikn|^2\right )^p
|(u,\gskmn^\star)|^2
\leq
C_p \|u\|^2_{\HH^p_k(\R^d)}.
\end{equation}
\end{proposition}

\begin{proof}
To lighten notations, let us write $\bz^{k, \bm, \bn} :=\xxkm + i \xikn$.
We have
\begin{align*}
&\sum_{[\bm, \bn]\in \Z^{2d}}
\left (|\xxkm|^2 + |\xikn|^2\right )^p
|(u,\gskmn^\star)|^2 \\
&= \sum_{[\bm, \bn]\in \Z^{2d}} 
|\bz^{k, \bm, \bn}|^{2 p} 
\left| \sum_{[\bm', \bn']\in \Z^{2d}}  (u,  \gskmnp) ( \gskmn^\star, \gskmnp^\star) \right|^2~~~~\text{ thanks to (\ref{eq:DeuxiemeDecompoPsiStar})}\\
&= \sum_{[\bm, \bn]\in \Z^{2d}} 
\left| \sum_{[\bm', \bn']\in \Z^{2d}} \left (1+ |\bz^{k, \bm', \bn'}|^p\right)  \frac{|\bz^{k, \bm, \bn}|^p}{1+ |\bz^{k, \bm', \bn'}|^p}  (u,  \gskmnp)  ( \gskmn^\star, \gskmnp^\star) \right|^2\\
&\leq \sum_{[\bm, \bn]\in \Z^{2d}} 
\left( \sum_{[\bm', \bn']\in \Z^{2d}} \left|\left (1+ |\bz^{k, \bm', \bn'}|^p\right)  (u,  \gskmnp)\right|^2 |( \gskmn^\star, \gskmnp^\star)| \right) \\
&\times \left( \sum_{[\bm', \bn']\in \Z^{2d}}  \left| \frac{|\bz^{k, \bm, \bn}|^p}{1+ |\bz^{k, \bm', \bn'}|^p} \right|^2 |( \gskmn^\star, \gskmnp^\star)|  \right), 
\end{align*}
thanks to Cauchy-Schwarz inequality. Let us bound the second factor using \eqref{eq:PSStar}.
For every $[\bm,\bn]\in \Z^{2d}$, we have
\begin{align*}
&
\sum_{[\bm', \bn']\in \Z^{2d}}
\left| \frac{|\bz^{k, \bm, \bn}|^p}{1+ |\bz^{k, \bm', \bn'}|^p} \right|^2
|( \gskmn^\star, \gskmnp^\star)|
\\
&\leq
C_p
\sum_{[\bm', \bn']\in \Z^{2d}}
\left|
\frac{|\bz^{k, \bm', \bn'}|^p + |\bz^{k, \bm, \bn}-\bz^{k, \bm', \bn'}|^p}%
{1+ |\bz^{k, \bm', \bn'}|^p}
\right|^2
e^{-| [\bm, \bn] - [\bm', \bn'] |^{1/2}}
\\
&\leq
C_p
\sum_{[\bm', \bn']\in \Z^{2d}}
\left| 1+  |\bz^{k, \bm, \bn}-\bz^{k, \bm', \bn'}|^p \right|^2
e^{-| [\bm, \bn] - [\bm', \bn'] |^{1/2}}
\\
&\leq
C_p.
\end{align*}

Therefore, we have
\begin{align*}
&
\sum_{[\bm, \bn]\in \Z^{2d}}
\left (|\xxkm|^2 + |\xikn|^2\right )^p |(u,\gskmn^\star)|^2
\\
&\leq
C_p
\sum_{[\bm, \bn]\in \Z^{2d}} \sum_{[\bm', \bn']\in \Z^{2d}}
\left |
\left (1+ |\bz^{k, \bm', \bn'}|^p\right )  (u,\gskmnp)
\right |^2
e^{-|[\bm,\bn] - [\bm',\bn']|^{1/2}}
\\
&\leq
C_p
\sum_{[\bm', \bn']\in \Z^{2d}}
\left |\left (1+ |\bz^{k, \bm', \bn'}|^p\right ) (u,\gskmnp)\right |^2,
\end{align*}
and \eqref{eq_decay_star} follows from \eqref{eq_decay}.
\end{proof}



\section{Proof of Theorem \ref{theorem_approximation}}
\label{Sec:Proof} 

\subsection{Approximation in $L^2(\R^d)$}

We start by proving Theorem \ref{theorem_approximation} in the simplest case, when $p=0$.

\begin{lemma}[$L^2(\R^d)$ approximation]
\label{Lem:ApproxL2}
For all $u \in \HH^r(\R^d)$, we have
\begin{equation}
\label{eq_L2_error_dual}
\|u-\Pi_D u\|_{L^2(\R^d)}
\leq
C_r D^{-r} \|u\|_{\HH^r_k(\R^d)}.
\end{equation}
\end{lemma}

\begin{proof}
By definition of $\Pi_D$, we have
\begin{equation*}
\CE
\eq
u-\Pi_D u
=
\sum_{\substack{[\bm,\bn] \in \Z^{2d} \\ |(\xxkm,\xikn)| > D}}
(u,\gskmn^\star) \gskmn,
\end{equation*}
and the frame property ensures that
\begin{align*}
\|\CE\|_{L^2(\R^d)}^2
&\leq
\sum_{\substack{[\bm,\bn] \in \Z^{2d} \\ |(\xxkm,\xikn)| > D}}
|(u,\gskmn^\star)|^2
\\
&\leq
D^{-2r}
\sum_{\substack{[\bm,\bn] \in \Z^{2d} \\ |(\xxkm,\xikn)| > D}}
|(\xxkm,\xikn)|^{2r} |(u,\gskmn^\star)|^2,
\end{align*}
and \eqref{eq_L2_error_dual} follows from Proposition \ref{proposition_decay_star}.
\end{proof}

\subsection{Technical preliminary results}

To prove Theorem \ref{theorem_approximation}, that is, to generalize \eqref{eq_L2_error_dual}
to the case where $p > 0$, we must first show several preliminary technical results.

\begin{lemma}[Smallness far from the origin]
For all $\ba \in \N^d$, we have
\begin{equation}
\label{eq_estimate_far}
|\partial^{\ba}(\Pi_D u)(\bx)|
\leq
C_{\ba}
k^{d/4}
\left(1+k^{1/2}D\right)^{d/2}
(k^{1/2}|\bx|)^{[\ba]}
e^{-\frac{k}{8} |\bx|^2} \|u\|_{L^2(\R^d)}
\end{equation}
for all $\bx \in \R^d$ with $|\bx| \geq 2D$.
In addition, for all $q \geq 0$, we have
\begin{equation}
\label{eq_L2_norm_far}
k^{q/2}
\left\||\bx|^q \partial^{\ba}(\Pi_D u)\right\|_{L^2(\{|\bx| > 2D\})}
\leq
C_{p,q} k^{[\ba]/2} e^{-\frac{k D^2}{16}} \|u\|_{L^2(\R^d)}.
\end{equation}
\end{lemma}

\begin{proof}
If $|\bx|\geq 2D$ and $|\xxkm| \leq D$, we have $|\bx-\xxkm|^2 \geq \frac{|\bx|^2}{4}$,
so that $e^{-\frac{k}{2} |\bx-\xxkm|^2} \leq e^{- \frac{k}{8} |\bx|^2}$. It follows from
\eqref{eq_diff_gskmn} that there exists a polynomial $Q_{\ba}$ of total degree $[\ba]$ such that
\begin{equation*}
(\partial^{\ba} \gskmn)(\bx)
= k^{[\ba]/2}
Q_{\ba}(k^{1/2}\left(\bx-\xxkm+i\xikn) \right) \gskmn(\bx)
\end{equation*}
so that
\begin{equation*}
|\partial^{\ba}(\Pi_D u)(\bx)|
\leq
C k^{d/4} \sum_{\underset{|(\xxkm,\xikn)| \leq D}{[\bm,\bn] \in \Z^{2d}}}
| (u,\gskmn^\star)||Q_{\ba}(k^{1/2}(\bx-\xxkm+i\xikn))| e^{-\frac{k}{8} |\bx|^2}.
\end{equation*}

We then observe that since $|\bx| \geq 2D$ and $|x_j^{k,m}|+|\xi_j^{k,n}| \leq CD$
for all terms in the sum, we have
\begin{equation*}
|\bx-\xxkm+i\xikn| \leq C|\bx|.
\end{equation*}
As a result, we have
\begin{equation*}
|\partial^{\ba}(\Pi_D u)(\bx)|
\leq
C_{\ba} k^{d/4}
\left (
\sum_{\underset{|(\xxkm,\xikn)| \leq D}{[\bm,\bn] \in \Z^{2d}}}
| (u,\gskmn^\star)|
\right )
(k^{1/2}|\bx|)^{[\ba]} e^{-\frac{k}{8} |\bx|^2}
\end{equation*}
and \eqref{eq_estimate_far} follows from \eqref{eq_frame} and from the Cauchy-Schwarz inequality,
as the sum is over a set of cardinal smaller than $C\left(1+k^{1/2}D\right)^d$.

Next, we set $p = [\ba]$, and we consider
\begin{equation*}
I_q
\eq
\int_{|\bx| \geq 2D}
\left (|\bx|^{p+q} e^{- \frac{k}{8} |\bx|^2}\right )^2
\mathrm{d}\bx
=
C \int_{2D}^{+\infty} r^{d-1+2(p+q)} e^{- \frac{k}{4} |\bx|^2}\mathrm{d}r,
\end{equation*}
using radial coordinates. Introducing $s = \sqrt{k}r$, we get
\begin{align*}
I_q
&=
\int_{2k^{1/2}D}^{+\infty}
(k^{-1/2} s)^{d-1+2(p+q)}
e^{-\frac{s^2}{4}} k^{-1/2} \mathrm{d}s
\\
&=
k^{-(d+2(p+q))/2}
\int_{2k^{1/2}D}^{+\infty} s^{d-1+2(p+q)}  e^{-\frac{s^2}{4}} \sqrt{s} \mathrm{d}s
\\
&\leq
k^{-(d+2(p+q))/2} e^{-\frac{k D^2}{2}}
\int_{0}^{+\infty} s^{d-1+2(p+q)}   e^{-\frac{s^2}{8}} \mathrm{d}s
\\
&=
C_p k^{-(d+2(p+q))/2} e^{-\frac{k D^2}{2}}.
\end{align*}
It then follows from \eqref{eq_estimate_far} that
\begin{align*}
\||\bx|^q \partial^{\ba}(\Pi_D u)\|_{L^2(\{|\bx| > 2D\})}
&\leq
C k^{d/4} (1+k^{1/2}D)^{d/2} k^p\sqrt{I_q} \|u\|_{L^2(\R^d)}
\\
&\leq
C_q
k^{d/4} (1+k^{1/2}D)^{d/2}
k^p
k^{-(d+2(p+q))/4} e^{-\frac{k D^2}{8}}
\|u\|_{L^2(\R^d)}
\\
&=
C_q
k^{-q/2} k^{p/2}(1+k^{1/2}D)^{d/2} e^{-\frac{k D^2}{8}}
\|u\|_{L^2(\R^d)}
\\
&\leq
C_q k^{-q/2} k^{p/2} e^{-\frac{k D^2}{16}}
\|u\|_{L^2(\R^d)},
\end{align*}
and \eqref{eq_L2_norm_far} follows.
\end{proof}

We then prove a preliminary result about the decay properties of $u - \Pi_D u$.

\begin{corollary}
\label{corollary_error_xL2}
For any $p,q,r \in \N$, any $\ba \in \N^d$ with $q+[\ba] \leq p$
and for any $u \in \HH^{p+r}(\R^d)$, we have
\begin{multline}
\label{eq_error_xL2}
k^{-[\ba]}
\||\bx|^q \partial^{\ba}\left(u-\Pi_D u\right)\|_{L^2(\R^d)}
\leq
\\
C_{\ba,q,r} D^q \left (
k^{-[\ba]}
\left\|\partial^{\ba}(u-\Pi_D u) \right\|_{L^2(\R^d)}
+
D^{-(p+r)} \|u\|_{\HH^{p+r}_k(\R^d)}
\right ).
\end{multline}

\end{corollary}
\begin{proof}
We first write that
\begin{align*}
&\left\||\bx|^q \partial^{\ba} \left(u-\Pi_D u\right)\right\|_{L^2(\R^d)}
\\
&=
\left\||\bx|^q \partial^{\ba} \left(u-\Pi_D u\right)\right\|_{L^2(\{|\bx| < 2D\})}
+
\left\||\bx|^q \partial^{\ba} \left(u-\Pi_D u\right)\right\|_{L^2(\{|\bx| > 2D\})}
\\
&\leq
(2D)^q \left\| \partial^{\ba}( u-\Pi_D u) \right\|_{L^2(\R^d)}
+
\||\bx|^q \partial^{\ba} u\|_{L^2(\{|\bx| > 2D\})}
+
\||\bx|^q \partial^{\ba}( \Pi_D u)\|_{L^2(\{|\bx| > 2D\})}.
\end{align*}
To deal with the second term, we note that
\begin{align*}
\||\bx|^q \partial^{\ba} u\|_{L^2(\{|\bx| > 2D\})}
\leq
(2D)^{q-p-r} \||\bx|^{q-q+p}+r \partial^{\ba} u\|_{L^2(\R^d)}
\leq
(2D)^{q-p-r} k^{[\ba]} \|u\|_{\HH_k^{p+r}(\R^d)}.
\end{align*}
On the other hand, for the last term, we use \eqref{eq_L2_norm_far}, to obtain
\begin{align*}
k^{-[\ba]} \||\bx|^q \partial^{\ba}( \Pi_D u)\|_{L^2(\{|\bx| > 2D\})}
&\leq
k^{(-[\ba]-q)/2} e^{-\frac{k D^2}{16}} \|u\|_{L^2(\R^d)}
\\
&\leq
C_r k^{(-[\ba]-q)/2} k^{2(q-p-r)} D^{q-p-r} \|u\|_{L^2}
\\
&\leq
C_r D^{(q-p)-r} \|u\|_{L^2}.
\end{align*}
Here, we used the fact that $e^{-\frac{x^2}{16}} \leq C_r x^{-r}$ for all $x \geq 1$,
and the fact that $k \geq 1$. The result follows.
\end{proof}

\subsection{Approximation in $H^p(\R^d)$}

The next key result needed to establish Theorem \ref{theorem_approximation}
concerns approximation in the space $H^p(\R^d)$.

\begin{lemma}[$H^p(\R^d)$ approximation]
For all $p,r \geq 0$, if $u \in \HH^{p+r}(\R^d)$, we have
\begin{equation}
\label{eq_error_diff}
k^{-p}\|\partial_j^p(u-\Pi_D u)\|_{L^2(\R^d)}
\leq
C_{p,r} D^{-r} \|u\|_{\HH^{p+r}_k(\R^d)}.
\end{equation}
\end{lemma}

\begin{proof}
Consider $u \in \HH^{p+r}(\R^d)$, fix $D > 0$ and let $\CE \eq u-\Pi_D u$, so that
\begin{equation*}
\CE = \sum_{|(\xxkm,\xikn)| > D} (u,\gskmn^\star) \gskmn.
\end{equation*}
Recall that the family $(\gskmn)_{[\bm,\bn] \in \Z^{2d}}$ is a frame, so that
\begin{equation*}
k^{-2p} \|\partial_j^p \CE\|^2_{L^2(\R^d)}
\leq
C
k^{-2p}
\sum_{[\bm,\bn] \in \Z^{2d}} |(\partial^p_j \CE,\gskmn)|^2
=
C(\sigma_{\rm small} + \sigma_{\rm large})
\end{equation*}
with
\begin{equation*}
\sigma_{\rm small}
\eq
k^{-2p} \sum_{|(\xxkm,\xikn)| \leq 2D} |(\partial^p_j \CE,\gskmn)|^2,
\quad
\sigma_{\rm large}
\eq
k^{-2p} \sum_{|(\xxkm,\xikn)| > 2D} |(\partial^p_j \CE,\gskmn)|^2.
\end{equation*}

We start by estimating the first term. Recalling \eqref{eq_diff_gskmn}
from Lemma \ref{lemma_differentiation}, we have
\begin{align*}
|(\partial^p_j \CE,\gskmn)|^2
=
|(\CE,\partial^p_j \gskmn)|^2
=
k^p\left|\left(\CE,Q_p\big(\sqrt{k} (x_j^{k,\bm} - x_j + i \xi_j^{k, \bn})\big)\gskmn\right)\right|^2.
\end{align*}
The binomial theorem yields that 
\begin{align*}
|Q_p(k^{1/2}(x_j^{k,\bm}-x_j+i\xi_j^{k,\bn}))|^2
&\leq
C_p \sum_{q=0}^p (k^{1/2}|x_j^{k,\bm}+i\xi_j^{k,\bn}|)^{2(p-q)} (k^{1/2}x_j)^{2q}
\\
&\leq
C_p \sum_{q=0}^p (k^{1/2} D)^{2(p-q)} k^q |\bx|^{2q},
\end{align*}
so that
\begin{equation*}
|(\partial_j^p \CE,\gskmn)|^2
\leq
C_p k^p \sum_{q=0}^p (k^{1/2} D)^{2(p-q)} k^q |(|\bx|^q \CE,\gskmn)|^2
\end{equation*}
for all $[\bm,\bn] \in \Z^{2d}$ with $|(\xxkm,\xikn)| \leq 2D$.
Since the coherent states form a frame, after summation over
$[\bm,\bn]$, we may write
\begin{align*}
\sigma_{\rm small}
&\leq
C_p k^{-p}
\sum_{q=0}^p (k^{1/2} D)^{2(p-q)} k^q \left\||\bx|^q \CE\right\|_{L^2(\R^d)}^2
\\
&=
C_p  D^{2p} \sum_{q=0}^p D^{-2q}  \left\||\bx|^q \CE\right\|_{L^2(\R^d)}^2.
\end{align*}
Since $q \leq p$, we may employ Corollary \ref{corollary_error_xL2} with $\ba =0$,
leading to
\begin{align*}
\sigma_{\rm small}
&\leq
C_{p,r} D^{2p}
\sum_{q=0}^p
\left (
\|\CE\|_{L^2(\R^d)}^2
+
D^{-2(p+r)} \|u\|_{\HH_k^{q+r}(\R^d)}^2
\right )
\\
&\leq
C_{p,r} \left (D^{2p} \|\CE\|_{L^2(\R^d)}^2
+
D^{-2r} \|u\|_{\HH_k^{p+r}(\R^d)}^2
\right ), 
\end{align*}
and it follows from \eqref{eq_L2_error_dual} that
\begin{equation*}
\sigma_{\rm small}
\leq
C_{p,r}  D^{-2r} \|u\|_{\HH_k^{p+r}(\R^d)}^2,
\end{equation*}
which is the desired result.

We now turn $\sigma_{large}$, that we further split as
\begin{equation*}
\sigma_{\rm large}
\leq
2 (\sigma_{\rm large}' + \sigma_{\rm large}''),
\end{equation*}
where
\begin{align*}
\sigma_{\rm large}'
&\eq
k^{-2p} \sum_{|(\xxkm,\xikn)| > 2D} |(\partial^p_j u,\gskmn)|^2
\\
\sigma_{\rm large}''
&\eq
k^{-2p} \sum_{|(\xxkm,\xikn)| > 2D} |(\partial^p_j\Pi_D u,\gskmn)|^2.
\end{align*}

We handle the first component with Proposition \ref{proposition_decay}.
Indeed, it follows from \eqref{eq_decay} that
\begin{align*}
\sigma_{\rm large}'
&\leq
k^{-2p} \sum_{|(\xxkm,\xikn)| > 2D} |(\partial^p_j u,\gskmn)|^2
\\
&\leq
C_r k^{-2p} D^{-2r} \sum_{|(\xxkm,\xikn)| > 2D} |(\xxkm,\xikn)|^{2r} |(\partial^p_j u,\gskmn)|^2
\\
&\leq
C_r k^{-2p} D^{-2r}\|\partial^p_j u\|_{\HH_k^r(\R^d)}^2
\leq
C_r D^{-2r}\|u\|_{\HH_k^{p+r}(\R^d)}^2.
\end{align*}
For the second component, we invoke Lemma \ref{lemma_L2_products}. On the one hand,
recalling the definition of $\Pi_D u$ in \eqref{eq_definition_pi}, we have
\begin{multline*}
|(\partial_j^p(\Pi_D u),\gskmn)|^2
=
\left | \sum_{|(\xxx{k}{\bm'},\xxi{k}{\bn'})| \leq D} 
(u, \gskmnp^\star) (\partial_j^p \gskmnp,\gskmn)\right |^2
\\
\leq
C (k^{1/2} D)^{2d}
\sum_{|(\xxx{k}{\bm'},\xxi{k}{\bn'})| \leq D} 
|(u, \gskmnp^\star)|^2 |(\partial_j^p \gskmnp,\gskmn)|^2,
\end{multline*}
since the summation happens on a set of cardinal less than $C (k^{1/2} D)^d$.
On the other hand, by Lemma \ref{lemma_L2_products},
\begin{multline*}
|(\partial_j^p \gskmn,\gskmnp)|
\leq
C_p k^{p/2} \left ( 1 + |\bn|^p \right )
e^{-\frac{\pi}{8} k D^2}
\\
=
C_p k^{p/2} \left ( 1 + (k^{1/2}D)^p \right ) e^{-\frac{\pi}{8} k D^2}
\leq
C_{p,r} k^{p/2} (k^{1/2}D)^{-r-d}
\leq
C_{p,r} k^{p/2} (k^{1/2}D)^{-d} D^{-r}
\end{multline*}
since $k \geq 1$.

Using the frame property, we obtain
\begin{align*}
\sigma_{\rm large}''
&\leq
k^{-2p} \left (
C_{p,r} k^{p} D^{-2r}
\sum_{|(\xxx{k}{\bm'},\xxi{k}{\bn'})| \leq D}  |(u,\gskmnp^\star) |^2
\right )
\\
&\leq 
C_{p,r} D^{-2r} \|u\|^2_{L^2(\R^d)}
\\
&\leq
C_{p,r} D^{-2r} \|u\|_{\HH^{p+r}_k(\R^d)}^2.
\end{align*}
\end{proof}

We are now ready to conclude the proof of our main result.

\begin{proof}[Proof of Theorem \ref{theorem_approximation}]
As previously, we set $\CE \eq u - \Pi_D u$. Using Fourier transform,
we start by observing that
\begin{equation*}
\|\partial^{\ba} \CE\|_{L^2(\R^d)}^2
=
\|(i\bxi)^{\ba} \CF \CE\|_{L^2(\R^d)}^2
\leq
\||\bxi|^{[\ba]} \CF \CE\|_{L^2(\R^d)}^2
\leq C
\sum_{j=1}^d \|\partial_j^{[\ba]} \CE\|_{L^2(\R^d)}^2.
\end{equation*}
As a result, using \eqref{eq_error_diff} (with $p=[\ba]$ and $r = (p+r)-[\ba]$), we have
\begin{equation*}
k^{-[\ba]} \|\partial^{\ba} \CE\|_{L^2(\R^d)}
\leq
C_{p,r}  D^{[\ba]-(p+r)} \|u\|_{\HH_k^{p+r}(\R^d)}.
\end{equation*}
for all $\ba \in \N^d$ with $[\ba] \leq p$.

Let $q\leq p$, and let $\ba \in \N^d$ with $\textcolor{violet}{q}+|\ba| \leq p$. Using \eqref{eq_error_xL2}, we have
\begin{align*}
&
k^{-[\ba]} \||\bx|^q \partial^{\ba} \CE\|_{L^2(\R^d)}
\\
&\leq
C_{p,r}
D^q
\left (
k^{-[\ba]} \|\partial^{\ba} \CE\|_{L^2(\R^d)}
+
D^{-(p+r)} \|u\|_{\HH_k^{p+r}(\R^d)}
\right )\\
&\leq
C_{p,r}
\left (
D^{[\ba]+q-p-r}
\|u\|_{\HH_k^{p+r}(\R^d)}
+
D^{(q-p)-r}
\|u\|_{\HH_k^{p+r}(\R^d)}
\right )
\\
&\leq
C_{p,r} 
D^{-r}
\|u\|_{\HH_k^{p+r}(\R^d)}
\end{align*}
since $q+[\ba] \leq p$ and $D \geq 1$, which concludes the proof.
\end{proof}

\appendix

\section{Elementary facts concerning Gaussian states}\label{App:Elem}

\subsection{Fourier transform and derivatives}

We start by recalling the expression of the Fourier transform of Gaussian states.

\begin{lemma}[Fourier transform]
For all $[\bm,\bn] \in \Z^{2d}$, we have
\begin{equation}
\label{eq_fourier_gskmn}
\CF_k(\gskmn) = e^{-ik\xxkm \cdot \xikn} \gs{k}{\bn}{-\bm}.
\end{equation}
\end{lemma}

\begin{proof}
Let $\by = k^{1/2}(\bx-\xxkm)$, we have
\begin{align*}
\CF_k(\gskmn)
&=
2^{d/2}\left (\frac{k}{2\pi}\right )^{3d/4}
\int_{\R^d}
e^{-\frac{k}{2}|\bx-\xxkm|^2}
e^{ik(\bx-\xxkm) \cdot \xikn}
e^{-ik\bx \cdot \bxi}
\dx
\\
&=
2^{d/2}\left (\frac{k}{2\pi}\right )^{3d/4}
e^{-i\xxkm \cdot \bxi}
\int_{\R^d}
e^{-\frac{|\by|^2}{2}}
e^{ik^{1/2} \by \cdot \xikn} e^{-ik^{1/2} \by \cdot \bxi} \dy
\\
&=
2^{d/2}\left (\frac{k}{2\pi}\right )^{3d/4}
e^{-i\xxkm \cdot \bxi}
\int_{\R^d}
e^{-\frac{|\by|^2}{2}}
e^{ik^{1/2} \by \cdot (\xikn-\bxi)} \dy
\\
&=
\left (\frac{k}{\pi} \right )^{d/4}
e^{-\frac{k}{2}|\bxi-\xikn|^2} e^{-i\xxkm \cdot \bxi}.
\end{align*}
\end{proof}

The following lemma gives an elementary expression for the derivatives of $\Psi_{k,\bm,\bn}$.

\begin{lemma}[Differentiation]
\label{lemma_differentiation}
For all $p \in \N$, there exists a polynomial $Q_p$ of degree $p$, with real coefficients,
and leading coefficient $z^p$, such that for all $j \in \{1,...,d\}$ and all
$[\bm,\bn]\in \Z^{2d}$, we have
\begin{equation}
\label{eq_diff_gskmn}
(\partial_j^p \gskmn) (\bx)
=
k^{p/2}Q_p\big(k^{1/2} (x_j^{k,\bm} - x_j + i \xi_j^{k, \bn})\big) \gskmn(\bx).
\end{equation}
\end{lemma}


\begin{proof}
We define the polynomials by induction, setting $Q_0 = 1$, and
\begin{equation*}
Q_{p+1}(z)  = Q_p'(z) + z Q_p(z),
\end{equation*}
for $p \geq 0$.
An elementary induction shows that $Q_p$ is a polynomial of degree $p$ whose term of
degree $p$ is $z^p$.

Our aim is now to show inductively that \eqref{eq_diff_gskmn} holds.
This is trivially the case for $p=0$ and for $p=1$, noting that
\begin{equation*}
\partial_j \gskmn
=
k\left (x_j^{k,\bm} - x_j + i \xi_j^{k, \bn} \right )
\gskmn.
\end{equation*}
Now, suppose the result holds for some $p\in \N$. We then have 
\begin{align*}
\partial_j^{p+1} \Psi_{k, \bm, \bn}
&=
k^{p/2} \Big[ k^{1/2} \partial_{j}Q_p(\sqrt{k}(x_j^{k,\bm} - x_j + i \xi_j^{k, \bn}))
\gskmn
\\
&+
Q_p(k^{1/2}(x_j^{k,\bm} - x_j + i \xi_j^{k,\bn}))
\times
k (x_j^{k,\bm} - x_j + i \xi_j^{k, \bn}) \gskmn \Big ]
\\
&=
k^{(p+1)/2} Q_{p+1}(k^{1/2}(x_j^{k,\bm}-x_j + i\xi_j^{k,\bn}))\gskmn,
\end{align*}
which proves the result.
\end{proof}

\subsection{Scalar products between Gaussian states}

We then provide expression and upper-bounds of inner-products involving
Gaussian states and their derivatives. 

\begin{proposition}[$L^2$ products]
\label{lemma_L2_products}
For all $[\bm,\bn],[\bm',\bn'] \in \Z^{2d}$, we have
\begin{equation}
\label{eq:NormePS}
|(\gskmn,\gskmnp)|
=
e^{-\frac{\pi}{4}|[\bm,\bn]-[\bm',\bn']|^2}.
\end{equation}
In addition, for all $p \in \mathbb N$, there exists
a constant $C_p$ such that
\begin{equation}
\label{eq:Eq:PSDerivFacile}
|(\partial_j^p \gskmn,\gskmnp)|
\leq
C_p k^{p/2} \left ( 1 + |\bn|^p \right )
e^{-\frac{\pi}{8}|[\bm,\bn]-[\bm',\bn']|^2}
\end{equation}
for all $[\bm,\bn],[\bm',\bn'] \in \Z^{2d}$.
\end{proposition}

To prove this proposition, we will need the following elementary result.

\begin{lemma}[A useful identity]
\label{proposition_panda}
Let $\bx,\by,\bz \in \R^d$. We have
\begin{equation}
\label{eq_panda_identity}
\frac{1}{2} \left (
|\bx-\by|^2 + |\bx-\bz|^2
\right )
=
\left |\bx - \frac{1}{2}(\by+\bz)\right |^2 + \left |\frac{1}{2}(\by-\bz) \right |^2.
\end{equation}
\end{lemma}

\begin{proof}
On the one hand, we have
\begin{equation*}
\frac{1}{2}\left (|\bx-\by|^2 + |\bx-\bz|^2\right )
=
|\bx|^2 -\bx \cdot (\by+\bz) + \frac{1}{2}|\by|^2 + \frac{1}{2}|\bz|^2.
\end{equation*}
On the other hand, we have
\begin{equation*}
\left |\bx - \frac{1}{2}(\by+\bz)\right |^2 + \left |\frac{1}{2}(\by-\bz) \right |^2
=
|\bx|^2 - \bx \cdot (\by+\bz) + \frac{1}{4}\left (|\by+\bz|^2+|\by-\bz|^2\right )
\end{equation*}
and the proof follows since
\begin{equation*}
|\by+\bz|^2 + |\by-\bz|^2 = 2|\by|^2 + 2|\bz|^2.
\end{equation*}
\end{proof}


Recall that, if we write, for $k > 0$ and $\bx^\star \in \R^d$
\begin{equation*}
g_{k,\bx^\star}(\bx) \eq e^{-k|\bx-\frac{1}{2}\bx^\star|^2},
\end{equation*}
then we have for all $\bxi^\star \in \R^d$
\begin{equation}\label{TFGaussienne}
\CF_k(g_{k,\bx^\star})(\bxi^\star)
=
2^{-d/2} e^{-\frac{k}{4}|\bxi^\star|^2} e^{-\frac{ik}{2}\bx^\star \cdot \bxi^\star}.
\end{equation}

We will use the following notations
\begin{align*}
\bx^{k,\bm,\bm'}&:= \frac{\xxkm+\xxkmp}{2}\\
\widehat{\bx}^{k,\bm,\bm'}&:= \xxkm-\xxkmp\\
\widehat{\bxi}^{k,\bn,\bn'}&:= \xikn-\xiknp.
\end{align*}

\begin{proof}[Proof of Proposition \ref{lemma_L2_products}]
We have
\begin{align*}
\gskmn(\bx)\overline{\gskmnp(\bx)}
&=
\left (\frac{k}{\pi}\right )^{d/2}
e^{-\frac{k}{2}|\bx-\xxkm |^2} e^{ ik(\bx-\xxkm )\cdot\xikn }
e^{-\frac{k}{2}|\bx-\xxkmp|^2} e^{-ik(\bx-\xxkmp)\cdot\xiknp}
\\
&=
\left (\frac{k}{\pi}\right )^{d/2}
e^{-\frac{k}{2}(|\bx-\xxkm |^2+|\bx-\xxkmp|^2)}
e^{ik\bx \cdot\widehat{\bxi}^{k,\bn,\bn'}}
e^{ik\xxkmp\cdot \widehat{\bxi}^{k,\bn,\bn'}}.
\end{align*}

Using \eqref{eq_panda_identity}, we have
\begin{equation*}
-\frac{k}{2}\left ( |\bx-\xxkm |^2+|\bx-\xxkmp|^2) \right )
=
-k \left |
\bx- \bx^{k,\bm,\bm'}
\right |^2
-
\frac{k}{4}|\widehat{\bx}^{k,\bm,\bm'}|^2,
\end{equation*}
leading to
\begin{multline}
\label{tmp_gaussian_product}
\gskmn(\bx)\overline{\gskmnp(\bx)}
=
\left (\frac{k}{\pi}\right )^{d/2}
e^{ik(\xxkmp\cdot\xiknp-\xxkm\cdot\xikn)}
e^{-\frac{k}{4}|\widehat{\bx}^{k,\bm,\bm'}|^2}
e^{-k|\bx- \bx^{k,\bm,\bm'}|^{2}}
e^{ik\bx \cdot\widehat{\bxi}^{k,\bn,\bn'}}.
\end{multline}
Recalling \eqref{TFGaussienne}, we have
\begin{equation}\label{eq:LesCrepesCestPasMalNonPlus}
\begin{aligned}
&(\gskmn,\gskmnp)
\\
&=
2^{d/2}
e^{ik(\xxkmp\cdot\xiknp-\xxkm\cdot\xikn)}
e^{-\frac{k^2}{4}|\widehat{\bx}^{k,\bm,\bm'}|^2}
\CF_k(g_{k,  2\bx^{k,\bm,\bm'}})(-\widehat{\bxi}^{k,\bn,\bn'})
\\
&=
e^{-i\bx^{k,\bm,\bm'} \cdot \widehat{\bxi}^{k,\bn,\bn'}}
e^{ik(\xxkmp\cdot\xiknp-\xxkm\cdot\xikn)}
e^{-\frac{k^2}{4}|\widehat{\bx}^{k,\bm,\bm'}|^2}
e^{-\frac{k^2}{4}|\widehat{\bxi}^{k,\bn,\bn'}|^2},
\end{aligned}
\end{equation}
which gives \eqref{eq:NormePS}.

Noting that $\bx^{k,\bm} = \bx^{k,\bm,\bm'} + \frac{\widehat{\bx}^{k,\bm,\bm'}}{2} $
and using Lemma \ref{lemma_differentiation}, we get 

\begin{align*}
\partial_j^p \Psi_{k, \bm, \bn} (\bx) &=  k^{p/2}Q_p\left(\sqrt{k} (x_j^{k,\bm,\bm'} - x_j + \frac{\widehat{x}^{k,\bm,\bm'}_j }{2} + i \xi_j^{k, \bn})\right) \Psi_{k, \bm, \bn}(\bx)\\
&= 
 \sum_{\ell=0}^p  k^{\frac{p+\ell}{2}} \sum_{\ell'=0}^\ell c_{\ell,\ell'} \left(\frac{\widehat{x}^{k,\bm,\bm'}_j }{2}+i\xi_j^{k,\bn}\right)^{\ell-\ell'} \left( x_j^{k,\bm,\bm'} - x_j \right)^{\ell'} \Psi_{k, \bm, \bn}(\bx),
\end{align*}
with all the coefficients $c_{\ell,\ell'}$ independent of $k$. Therefore, we have 
\begin{multline}
\label{eq:MaxENAMarreDeCalculer}
\left|(\partial^p_j \gskmn,\gskmnp)\right|
\leq
\\
C_p
\sum_{\ell=0}^p
k^{\frac{p+\ell}{2}}
\sum_{\ell'=0}^\ell
\left(|\widehat{x}^{k,\bm,\bm'}_j|+|\xi_j^{k,\bn}|\right)^{\ell-\ell'}
\left|
\left(
\left( x_j^{k,\bm,\bm'} - x_j \right)^{\ell'} \gskmn, \gskmnp
\right)
\right|. 
\end{multline}

Recalling \eqref{tmp_gaussian_product}, we have
\begin{multline}
\label{tmp_gaussian_product_deriv}
 \left( x_j^{k,\bm,\bm'} - x_j \right)^{\ell'} \gskmn(\bx)\overline{\gskmnp(\bx)}
=
\\
\left (\frac{k}{\pi}\right )^{d/2}
e^{ik(\xxkmp\cdot\xiknp-\xxkm\cdot\xikn)}
e^{-\frac{k}{4}|\widehat{\bx}^{k,\bm,\bm'}|^2}
 \left( x_j^{k,\bm,\bm'} - x_j \right)^{\ell'}  e^{-k|\bx-\bx^{k,\bm,\bm'}|^2}
e^{ik\bx \cdot\widehat{\bxi}^{k,\bn,\bn'}}.
\end{multline}

We then compute

\begin{align*}
&\int_{\R^d}
(\bx-\bx^{k,\bm,\bm'})^{\ell'}_j e^{-k|\bx-\bx^{k,\bm,\bm'}|^2} e^{ik\bx \cdot \widehat{\bxi}^{k,\bn,\bn'}}
d\bx
\\
&=
e^{ik \bx^{k,\bm,\bm'} \cdot  \widehat{\bxi}^{k,\bn,\bn'} }
\int_{\R^d}
k^{\ell'/2} \by_j^\ell e^{-|\by|^2} e^{i\sqrt{k}\by \cdot \widehat{\bxi}^{k,\bn,\bn'}}
d\by~~~~\text{ setting $\by =k^{-1/2} (\bx-\bx^{k,\bm,\bm'})$}
\\
&=
e^{ik \bx^{k,\bm,\bm'} \cdot  \widehat{\bxi}^{k,\bn,\bn'} } (i\sqrt{k}\partial_j)^{\ell'} (e^{-\frac{1}{4}|\bxi|^2}) \left(\sqrt{k}\widehat{\bxi}^{k,\bn,\bn'}\right).
\end{align*}

Now, $\partial_j^{\ell'} (e^{-\frac{1}{4}|\bxi|^2})$ is a polynomial of degree $\ell'$ times $e^{-\frac{1}{4}|\bxi|^2}$, so that its modulus is bounded by $C_{\ell'} (1+ |\bxi|^{\ell'})  e^{-\frac{1}{4}|\bxi|^2}$. Therefore, we have
$$\left| \left( \left( x_j^{k,\bm,\bm'} - x_j \right)^{\ell'} \Psi_{k,\bm,\bn}, \Psi_{k, \bm',\bn'} \right) \right|\leq C k^{\frac{\ell'}{2}} |\widehat{\bxi}^{k,\bn,\bn'}|^{\ell'} e^{-\frac{k}{4}|\widehat{\bx}^{k,\bm,\bm'}|^2} e^{- \frac{k}{4} |\widehat{\bxi}^{k,\bn,\bn'}|^2}.$$

Combining this with (\ref{eq:MaxENAMarreDeCalculer}), we get
\begin{align*}
\left|(\partial^p_j \gskmn,\gskmnp)\right|&\leq
 C_p k^{\frac{p}{2}}  e^{-\frac{k}{4}|\widehat{\bx}^{k,\bm,\bm'}|^2} e^{- \frac{k}{4} |\widehat{\bxi}^{k,\bn,\bn'}|^2}  \sum_{\ell=0}^p  \sum_{\ell'=0}^\ell k^{\frac{\ell+\ell'}{2}}  \left(|\widehat{\bx}^{k,\bm,\bm'}|+|\bxi^{k,\bn}|\right)^{\ell-\ell'} |\widehat{\bxi}^{k,\bn,\bn'}|^{\ell'}\\
 &\leq C_p k^{\frac{p}{2}}  e^{-\frac{k}{4}|\widehat{\bx}^{k,\bm,\bm'}|^2} e^{- \frac{k}{4} |\widehat{\bxi}^{k,\bn,\bn'}|^2}  \sum_{\ell=0}^p  \left(|\sqrt{k} \widehat{\bx}^{k,\bm,\bm'}|+| \sqrt{k} \bxi^{k,\bn}| + |\sqrt{k} \widehat{\bxi}^{k,\bn,\bn'}| \right)^\ell\\
 &\leq  C_p k^{\frac{p}{2}}  e^{-\frac{k}{4}|\widehat{\bx}^{k,\bm,\bm'}|^2} e^{- \frac{k}{4} |\widehat{\bxi}^{k,\bn,\bn'}|^2}  \left[1+  \left(|\sqrt{k} \widehat{\bx}^{k,\bm,\bm'}|+| \sqrt{k} \bxi^{k,\bn}| + |\sqrt{k} \widehat{\bxi}^{k,\bn,\bn'}| \right)^p\right],
\end{align*}
as announced.

\end{proof}

\section{Localisation properties of the dual frame}
\label{App:DecayPSStar}

The aim of this appendix is to give an elementary proof of \eqref{eq:PSStar}.
Recalling the constants $\alpha$ and $\beta$ from the frame inequalities
in \eqref{eq_frame}, we set
\begin{equation}
\label{eq_definition_Y}
Y \eq I- \frac{2}{\alpha+\beta} T_k^\star T_k,
\end{equation}
so that $Y$ is self-adjoint with  $\|Y\|_{L^2(\R^d) \to L^2(\R^d)} \leq \gamma <1$.
In particular,
\begin{equation*}
S_k = (T_k^\star T_k)^{-1} = \frac{2}{\alpha+\beta} (I - Y)^{-1} =  \frac{2}{\alpha+\beta} \sum_{\ell=0}^\infty Y^\ell,
\end{equation*}
where the sum converges for the norm of bounded linear operators acting on $L^2(\R^d)$. Therefore, we have
\begin{align*}
(\gskmn^\star,\gskmnp^\star)
&=
\left (\frac{2}{\alpha+\beta}\right )^2
\left (
\sum_{\ell=0}^{+\infty} Y^\ell \gskmn,
\sum_{\ell'=0}^{+\infty} Y^{\ell'}\gskmnp
\right )
\\
&=
\left (\frac{2}{\alpha+\beta}\right )^2
\sum_{\ell, \ell'=0}^{+\infty}
\left (\gskmn, Y^{\ell+\ell'}\gskmnp \right )
\\
&=
\left (\frac{2}{\alpha+\beta}\right )^2
\sum_{p=0}^{+\infty}
(p+1) \left (\gskmn, Y^p\gskmnp \right ).
\end{align*}

We claim that there exists $A>1$ such that, for all $p\in \N$ and all
$[\bm, \bn]$, $[\bm', \bn']\in \Z^{2d}$, we have
\begin{equation}
\label{eq:RecY}
|\left(\gskmn,Y^{p} \gskmnp \right)|
\leq
A^p e^{-\frac{\pi}{4} |[\bm, \bn]- [\bm', \bn']|}.
\end{equation}
Let us prove this result by induction. It trivially holds for $p=0$,
due to \eqref{eq:NormePS}. Assume now that \eqref{eq:RecY} holds for $p-1$.
Recalling \eqref{eq_definition_Y}, we have
\begin{equation*}
Y^p = Y^{p-1} - \frac{2}{\alpha+\beta}  (T_k^\star T_k)Y^{p-1}
\end{equation*}
and its follows from \eqref{eq_TksTku} that
\begin{align*}
\left (\gskmn, Y^p \gskmnp \right )
&=
\left(\gskmn ,  Y^{p-1} \gskmnp \right)
\\
&-\frac{2}{\alpha+\beta}
\sum_{[\bm'', \bn'']\in \Z^{2d}}  \left(\gskmn, \gskmnpp\right )
\left(\gskmnpp ,  Y^{p-1} \gskmnp \right),
\end{align*}
so that, by our induction hypothesis and \eqref{eq:NormePS}, we have 
\begin{align*}
\left | \left ( \gskmn,  Y^{p} \gskmnp \right )\right |
&\leq
A^{p-1}e^{-\frac{\pi}{4} |[\bm, \bn]- [\bm', \bn']|}
\\
&+
\frac{2A^{p-1}}{\alpha+\beta}
\sum_{[\bm'', \bn'']\in \Z^{2d}}
e^{-\frac{\pi}{4} |[\bm, \bn]- [\bm'', \bn'']|^2}
e^{- \frac{\pi}{4} |[\bm', \bn']- [\bm'', \bn'']|}.
\end{align*}

We now simplify the sum. For the sake of simplicity, we introduce
$\bq = [\bm,\bn]$, $\bq' = [\bm',\bn']$ and $\bq'' = [\bm'',\bn'']$.
The triangle inequality reveals that
\begin{align*}
|\bq- \bq''|^2 + |\bq'- \bq''|
&=
|\bq- \bq''|^2 - |\bq-\bq''| + |\bq- \bq''|  + |\bq'- \bq''|
\\
&\geq
|\bq- \bq''|^2 - |\bq-\bq''| +  |\bq - \bq'|,
\end{align*}
and as a result
\begin{align*}
\sum_{\bq'' \in \Z^{2d}}
e^{-\frac{\pi}{4} |\bq- \bq''|^2}
e^{-\frac{\pi}{4} |\bq'-\bq''|}
\leq
e^{-\frac{\pi}{4}|\bq-\bq'|}
\sum_{\bq'' \in \Z^{2d}}
e^{\frac{\pi}{4} (|\bq-\bq''|-|\bq- \bq''|^2)}
\leq
\sigma e^{-\frac{\pi}{4} |\bq-\bq'|},
\end{align*}
where we easily see that
\begin{equation*}
\sigma
\eq
\sum_{\bq'' \in \Z^{2d}}
e^{\frac{\pi}{4} (|\bq-\bq''|-|\bq- \bq''|^2)}
\end{equation*}
is finite and does not depend on $\bq$.

Therefore, we have 
\begin{align*}
\left| \left (\gskmn, Y^{p} \gskmnp \right )\right|
&\leq
A^{p-1}\left (1 + \frac{2\sigma}{\alpha+\beta}\right )
e^{-\frac{\pi}{4} |[\bm, \bn]- [\bm', \bn']|}
\end{align*}
and \eqref{eq:RecY} holds, provided we take 
\begin{equation*}
A
\geq
1
+
\frac{2\sigma}{\alpha+\beta}.
\end{equation*}

To conclude, we write $\delta \eq |[\bm, \bn]-[\bm', \bn']|$,
and we let $P$ denote the largest integer such that $P \leq \delta^{1-\varepsilon}$.
We have, for any $\gamma' \in (\gamma, 1)$
\begin{align*}
|\left (\gskmn^\star,\gskmnp^\star) \right|
&\leq C
\sum_{p=0}^{P}
(p+1) \left|\left (\gskmn,  Y^{p}\gskmnp \right )\right|
+
C \sum_{p=P+1}^{+\infty} (p+1)\gamma^p
\\
&\leq C
e^{- \frac{\pi}{4} \delta}
\sum_{p=0}^{P} (p+1) A^p
+
C \sum_{p=P+1}^{+\infty} (\gamma')^p
\\
&\leq
C (P+1)^2A^P e^{-\frac{\pi}{4} \delta}
+
C (\gamma')^{P+1}
\\
&\leq
C\delta^2 A^{\delta^{1-\varepsilon}} e^{-\frac{\pi}{4}\delta}
+
C (\gamma')^{\delta^{1-\varepsilon}}.
\end{align*}
One readily checks that both terms are bounded by $C_\varepsilon e^{- \delta^{1-2\varepsilon}}$ for any $\varepsilon>0$,
which concludes the proof.

\section{Sharpness of Theorem \ref{theorem_approximation}}\label{App:Sharp}
\label{appendix_sharp}

The aim of this appendix is to show the sharpness of Theorem \ref{theorem_approximation},
as asserted in Proposition \ref{proposition_sharp}. We start with a preliminary result
concerning the norm of a coherent Gaussian state.

\begin{lemma}[$\HH^p_k(\R^d)$ norm]
\label{lemma_norm_gskmn}
Let $p\geq 1$. For all $[\bm, \bn] \in \Z^{2d}$, we have
\begin{equation}
\label{eq:NormeGauss}
\mu_p \left(1+ |\xxkm|^{2} +  |\xxi{k}{\bn}|^{2}  \right)^p
\leq
\|\Psi_{k,\bm,\bn}\|_{\HH^p_k(\R^d)}^2
\leq
\lambda_p\left(1+ |\xxkm|^{2} +  |\xxi{k}{\bn}|^{2}  \right)^p,
\end{equation}
where the constants $\mu_p,\lambda_p > 0$ only depend on $p$ and $d$.
\end{lemma}

\begin{proof}
Fix $[\bm,\bn] \in \Z^{2d}$ and let $j \in \{1,\dots,d\}$ be such that $n_j \geq [\bn]/d$.
We have
\begin{equation}\label{eq:TroisTermes}
\|\Psi_{k,\bm,\bn}\|_{\HH^p_k(\R^d)}^2
\geq
\|\Psi_{k,\bm,n} \|^2_{L^2(\R^d)}
+
\||\bx|^p \Psi_{k,\bm,n} \|^2_{L^2(\R^d)}
+
k^{-2p} \| \partial^{p}_{x_j} \Psi_{k,\bm,n}\|_{L^2(\R^d)}^2.
\end{equation}

The first term is equal to 1. The second is 
\begin{align*}
 \||\bx|^p \Psi_{k,\bm,n} \|^2_{L^2(\R^d)}
&=
 \left (\frac{k}{\pi} \right )^{d/2}
\int_{\R^d} |\bx|^{2p} e^{- k|\bx-\xxx{k}{\bm}|^2} \mathrm{d} \bx
\\
&=
 \pi^{-d/2}
\int_{\R^d} \left|\xxx{k}{\bm} + k^{-1/2} \by \right|^{2p} e^{- |\by|^2} \mathrm{d} \by
\geq
c k^p |\xxx{k}{\bm}|^{2p}.
\end{align*}

The last term in (\ref{eq:TroisTermes}) is equal to
\begin{align*}
k^{-2p}
 \| \partial^{p}_{x_j} \Psi_{k,\bm,n}\|_{L^2(\R^d)}^2
&= k^{-p}
\left\|
Q_{p} \left ( k^{1/2} (x_j^{k,\bm} -x_j  + i \xi_j^{k, \bn})  ) \right ) \Psi_{k,\bm,\bn}  \right\|_{L^2(\R^d)}^2 \\
&= k^{-p}  \pi^{-d/2}
\int_{\R^d} 
\left| Q_{p} (y_j + i k^{1/2} \xi_j^{k,\bn})\right|^2  e^{- |\by|^2} \mathrm{d} \by\\
&\geq (k^{-p}+  |\xxi{k}{\bn}|^{2p} ).
\end{align*}
This gives us the lower bound in \eqref{eq:NormeGauss}.

On the other hand, it holds that
\begin{align*}
\|\Psi_{k,\bm,\bn} \|_{\HH^p_k(\R^d)}^2
&= 
\sum_{[\ba] \leq p}
\sum_{q=0}^p
k^{-[\ba]} \left\|
|\bx|^q  \prod_{j=1}^d
Q_{a_j} \left ( k^{1/2} (x_j^{k,\bm} -x_j  + i \xi_j^{k, \bn})  ) \right ) \Psi_{k,\bm,\bn}  \right\|_{L^2(\R^d)}^2
\\
&=
\sum_{[\ba] \leq p}
\sum_{q=0}^p k^{-[\ba]} \pi^{-d/2}
\int_{\R^d} \left|\xxx{k}{\bm} + k^{-1/2} \by \right|^{2q}
\left|\prod_{j=1}^d Q_{a_j} (y_j + i k^{1/2} \xi_j^{k,\bn})\right|^2  e^{- |\by|^2} \mathrm{d} \by.
\end{align*}
We then observe that
\begin{multline*}
k^{-[\ba]}\int_{\R^d} \left|\xxx{k}{\bm} + k^{-1/2} \by \right|^{2q}
\left|\prod_{j=1}^d Q_{a_j} (y_j+ i k^{1/2} \xi_j^{k,\bn})\right|^2 e^{- |\by|^2} \mathrm{d} \by
\\
\leq
C_{d,p}  k^{-[\ba]} |\xxkm|^{2q} \prod_{j=1}^d \left(1+| k^{1/2} \xi_j^{k,\bn}|^{a_j}\right)^2
\leq
C_{d,p} \left(1+ |\xxkm|^{2} + |\xxi{k}{\bn}|^{2} \right)^{p},
\end{multline*}
which lead to the upper bound in \eqref{eq:NormeGauss} after summation.
\end{proof}

We then need another intermediate result concerning the inner product
of elements of the Gabor and dual frames.

\begin{lemma}[$L^2$ product of the dual frame]
\label{Lem:PSDerivPsiStar}
For all $[\bm,\bn], [\bm', \bn'] \in \Z^{2d}$, we have
\begin{equation}
\label{eq_diff_gskmn_star_gskmn}
\left|(\gskmn^\star,\gskmnp)\right|
\leq
C
e^{-\frac{\pi}{8} |[\bm,\bn]-[\bm',\bn']|^{1/2}}.
\end{equation}
\end{lemma}

\begin{proof}
Using \eqref{eq:Eq:PSDerivFacile} and \eqref{eq:PSStar}, we get
\begin{align*}
\left|(\gskmn^\star,\gskmnp)\right|
&\leq
\sum_{[\bm'',\bn''] \in \Z^{2d}}
|(\gskmn^\star,\gs{k}{\bm''}{\bn''}^\star)|
|( \gs{k}{\bm''}{\bn''},\gskmnp)|
\\
&\leq
C
\sum_{[\bm'',\bn''] \in \Z^{2d}}
e^{-|[\bm'',\bn'']-[\bm, \bn] |^{1/2}}
e^{-\frac{\pi}{8}|[\bm'',\bn'']-[\bm', \bn']|^2}
\\
&=
C
\sum_{[\bm'',\bn''] \in \Z^{2d}}
e^{-|\bq''-\bq|^{1/2}-\frac{\pi}{8}|\bq''-\bq'|^2},
\end{align*}
where we introduced
$\bq:= [\bm, \bn]$, $\bq':= [\bm', \bn']$ and $\bq'':= [\bm'', \bn'']$
for the sake of shortness. Since $\bq',\bq'' \in \Z^{2d}$, we have
$|\bq''-\bq'|^2 \geq |\bq''-\bq'|^{1/2}$, and hence
\begin{equation*}
e^{-|\bq''-\bq|^{1/2}-\frac{\pi}{8}|\bq''-\bq'|^2}
\leq
e^{-|\bq''-\bq|^{1/2}-\frac{\pi}{8}|\bq''-\bq'|^{1/2}}
=
e^{-\frac{\pi}{8}\left (|\bq''-\bq|^{1/2}+|\bq''-\bq'|^{1/2}\right )}
e^{-\left (1-\frac{\pi}{8}\right )|\bq''-\bq|^{1/2}},
\end{equation*}
where we note that $1-\pi/8 > 0$. Then, since
\begin{equation*}
\sqrt{|\bx-\by|}
\leq 
\sqrt{|\bx| + |\by|}
\leq
\sqrt{|\bx|}+ \sqrt{|\by|} \qquad \forall \bx,\by \in \R^{2d},
\end{equation*}
we have
\begin{equation*}
e^{-|\bq''-\bq|^{1/2}-\frac{\pi}{8}|\bq''-\bq'|^2}
\leq
e^{-\frac{\pi}{8}|\bq'-\bq|^{1/2}}
e^{-\left (1-\frac{\pi}{8}\right )|\bq''-\bq|^{1/2}}.
\end{equation*}
Therefore, we get
\begin{align*}
\left|(\gskmn^\star,\gskmnp)\right|
&\leq
C
e^{-\frac{\pi}{8}|\bq - \bq'|^{1/2}}
\sum_{[\bm'',\bn''] \in \Z^{2d}}
 e^{-(1-\frac{\pi}{8}) |\bq''- \bq |^{1/2}}
\\
&\leq
C e^{-\frac{\pi}{8} |\bq - \bq' |^{1/2}}
\end{align*}
as announced.
\end{proof}

\begin{proof}[Proof of Proposition \ref{proposition_sharp}]
Let $D \geq k^{-1/2}$. Pick $\bm = (2\lceil k^{1/2}D \rceil,0,0)$ and let
$u = \gs{k}{\bm}{\bzero}$. For any $p \geq 0$, from Lemma \ref{lemma_norm_gskmn}, we have
\begin{equation*}
\|u\|_{\HH^{p}_k(\R^d)} \geq \mu_{p} (2D)^{-p}.
\end{equation*}
On the other hand since
\begin{align*}
\Pi_D\gs{k}{\bm}{\bzero}
&=
\sum_{\substack{[\bm',\bn'] \in \Z^{2d} \\ |[\bm',\bn']| \leq k^{1/2} D}}
(\gs{k}{\bm}{\bzero},\gskmnp^\star) \gskmnp
\end{align*}
we have
\begin{align*}
\|\Pi_D\gs{k}{\bm}{\bzero}\|_{\HH^p_k(\R^d)}
&\leq
\sum_{\substack{[\bm',\bn'] \in \Z^{2d} \\ |[\bm',\bn']| \leq k^{1/2} D}}
|(\gs{k}{\bm}{\bzero},\gskmnp^\star)| \|\gskmnp\|_{\HH_k^p(\R^d)}
\\
&\leq
C
\sum_{\substack{[\bm',\bn'] \in \Z^{2d} \\ |[\bm',\bn']| \leq k^{1/2} D}}
e^{-\frac{\pi}{8}|\bm-\bm'|^{1/2}} \|\gskmnp\|_{\HH_k^p(\R^d)}
\\
&\leq C
\sum_{\substack{[\bm',\bn'] \in \Z^{2d} \\ |[\bm',\bn']| \leq k^{1/2} D}}
e^{-\frac{\pi}{8}|\bm-\bm'|^{1/2}}  D^p\\
&\leq C D^p
M(k^{1/2} D)^{2d+1} e^{-a \sqrt{k^{1/2} D}}\\
&\leq C D^p e^{-a' \sqrt{k^{1/2} D}}\\
&\leq C D^p e^{-a' \sqrt{D}},
\end{align*}
since $k\geq 1$.

In particular, there exists $D_*$ such that for all $D\geq D^*$, this quantity is smaller than $\mu_{p} D^{-p}$. We thus have, for all $D\geq D^*$
\begin{equation*}
\|u-\Pi_D u\|_{\HH_k^p(\R^d)}
\geq
\|u\|_{\HH^p_k(\R^d)}
-
\|\Pi_D u\|_{\HH^p_k(\R^d)}
\geq
\mu_{p} D^{-p},
\end{equation*}
and thanks to Lemma \ref{lemma_norm_gskmn}, this quantity is larger than $\frac{\mu_{p+r}}{\lambda_p} D^{-r} \|u\|_{\HH^{p+r}(\R^d)}$. The result follows.
\end{proof}

\section{Proof of Theorem \ref{theorem_approximation} using Modulation Spaces}\label{Sec:Modulation}

\subsection{Rescaling}

First of all, let us explain how the proof of Theorem \ref{theorem_approximation} can be reduced to the case where $k=1$.
Recall that the isometry $\delta_k$ was introduced in (\ref{eq:DefRescaling}). For any $s\in \N$, we write, for all $\ba \in \Z^d$ and $q\in \N$
with $[\ba] + q =s$
\begin{align*}
\| |\bx|^q \partial^{\ba} u\|_{L^2(\R^d)}
=
\| \delta_{k} (|\bx|^q \partial^{\ba} u)\|_{L^2(\R^d)}
&=
\| k^{-q/2} |\bx|^q \delta_{k} (\partial^{\ba} u)\|_{L^2(\R^d)}
\\
&=
\| k^{(-q+ \ba)/2}  |\bx|^q \partial^{\ba} (\delta_{k} u)\|_{L^2(\R^d)}
=
k^{-s/2} k^{[\ba]} \| |\bx|^q \partial^{\ba} \delta_{k} u\|_{L^2(\R^d)} .
\end{align*}

Therefore,  if we introduce the norm
$$\|u\|^2_{\widetilde{H}_k^s(\R^d)} \eq \sum_{\alpha+\beta = s} k^{-2|\beta|} \| x^\alpha \partial^\beta u\|_{L^2(\R^d)}^2,$$
for any parameter $k\geq 1$, we have
$$\|u\|_{\widetilde{H}_k^s(\R^d)}= k^{-s/2} \| \delta_{k} u\|_{\widetilde{H}_1^s(\R^d)}.$$

We want to study
\begin{align*}
\left \|u - \sum_{|(\bm,\bn)|\leq D \sqrt{k/\pi}} \left(u, \Psi^*_{k,\bm,\bn}\right) \Psi_{k,\bm,\bn}\right\|_{\widetilde{H}_k^s(\R^d)} &= k^{-\frac{s}{2}} \left\|\delta_{k} u - \sum_{|(\bm,\bn)|\leq D \sqrt{k/\pi}} \left(\delta_{k} u,  \delta_{k} \Psi^*_{k,\bm,\bn}\right) \delta_{k} \Psi_{k,\bm,\bn}\right\|_{\widetilde{H}_1^s(\R^d)}\\
 &= k^{-\frac{s}{2}}  \left\|\delta_{k} u - \sum_{|(\bm,\bn)|\leq D \sqrt{k/\pi}} \left(\delta_{k} u,   \Psi^*_{1,\bm,\bn}\right) \Psi_{1,\bm,\bn}\right\|_{\widetilde{H}_1^s(\R^d)}.
 \end{align*}

 Suppose that we can show that, for any $v\in \widehat{H}^{s+s'}$, we have
\begin{equation}\label{eq:K1}
 \left\|v - \sum_{|(\bm,\bn)|\leq D} \langle(v, \Psi^*_{1,\bm,\bn}\rangle \Psi_{1,\bm,\bn}\right\|_{\widetilde{H}_1^s}\leq   C(s,s') D^{-s'} \|v\|_ {\widetilde{H}_1^{s+s'}}.
\end{equation}

We will then deduce from what precedes that
\begin{align*}
\|u - \Pi_D u\|_{\HH^p_k(\R^d)} &= \sum_{s=0}^p \|u - \Pi_D u\|_{\widetilde{H}^s_k(\R^d)}\\
&\leq \sum_{s=0}^p C(s,s') k^{-\frac{s}{2}} D^{-s'} k^{-\frac{s'}{2}} \|\delta_k u\|_ {\widetilde{H}_1^{s+s'}}\\
& = D^{-s'}  \sum_{s=0}^p C(s,s')  \| u\|_ {\widetilde{H}_k^{s+s'}}\\
&\leq C D^{-s'} \|u\|_{H_k^{p+s'}},
\end{align*}
which gives us (\ref{eq_approximation}). Therefore, Theorem \ref{theorem_approximation} will follow if we prove (\ref{eq:K1}).

\subsection{Reminder on modulation spaces}

If $(\bx_0, \bxi_0)\in \R^{2d}$, we introduce the function on $\R^d$
\begin{equation}
\label{eq:DefGaussian}
\psi_{\bx_0,\bxi_0}(\bx)
\eq
\pi^{-d/4} e^{ i(\bx-\bx_0)\cdot \bxi_0} e^{-\frac{|\bx-\bx_0 |^2}{2}}
\quad \forall \bx \in \R^d.
\end{equation}
For every $s\in \N$, we introduce the weight $v_s$ on $\R^{2d}$ given by $v_s(\bx,\bxi) = (1+ |\bx|^2 + |\bxi|^2)^{s/2}$, and we define the corresponding modulation spaces $M^p_{v_s}$ for all $p\geq 1$ to be the space of functions such that the following norm is finite:
$$ \|f\|_{M^p_{v_s}} := \left(\int_{\R^{2d}} |v_s(\bx,\bxi)|^p \left| \left( f,  \psi_{\bx,\bxi} \right) \right|^p \mathrm{d}\bx \mathrm{d}\bxi\right)^{1/p}.$$

Let us discuss another interpretation of these spaces when $p=2$, following \cite[Proposition 11.3.1]{grochenig2001foundations}.
Let $\ba,\bb \in \Z^d$ with $[\ba] + [\bb]\leq s$, so that $|\bx^{\ba} \bxi^{\bb} | \leq v^2_s(\bx,\bxi)$. Writing $(\tau_{\bx} f)(\by) := f(\by-\bx)$ and $g(\bx) = e^{-|\bx|^2/2}$,  we then have
\begin{align*}
 \|f\|^2_{M^2_{v_s}} &\geq \int_{\R^{d}} |\bx^{\ba}| \left(\int_{\R^d} |\bxi^{\bb}| \left|\left( e^{i \bx\cdot \bxi} \tau_{\bx} f, g   \right) \right|^2\right) \mathrm{d} \bx\\
 &= \int_{\R^{d}} |\bx^{\ba}| \left(\int_{\R^d} |\bxi^{\bb}| \left| (\mathcal{F} (\tau_{\bx} f, g))(\bxi)  \right|^2\right) \mathrm{d} \bx\\
 &= \int_{\R^{d}} |\bx^{\ba}| \left\| \partial_{\bx}^{\bb} (\tau_{\bx} f, g) \right\|_{L^2}^2 \mathrm{d}\bx\\ 
 &= \int_{\R^{d}}\int_{\R^{d}}  |\bx^{\ba}| |\partial_{\bx}^{\bb} f|^2(\bx-\by) |g(\by)|^2 \mathrm{d}\by \mathrm{d}\bx\\
 &= \int_{\R^d}  |\partial_{\bz}^{\bb} f|^2(\bz) \left(\int_{\R^{d}}  |\by^{\ba} + \bz^{\ba}| |g(\by)|^2 \mathrm{d}\bz \right) \mathrm{d}\by\\
 &\geq C \int_{\R^d}  |\partial_{\bz}^{\bb} f|^2(\bz) |\bz^{\ba}| \mathrm{d}\bz
\end{align*}

Conversely,  \cite[Proposition 11.3.1]{grochenig2001foundations} implies that $ \|f\|_{M^2_{v_s}}\leq C \|f\|_{\widetilde{H}_1^s(\R^d)}$, so that $\|\cdot \|_{M^2_{v_s}}$ and $\| \cdot \|_{\widetilde{H}_1^s(\R^d)}$ are equivalent norms on $M_{v_s}^2$.

\subsection{Sequence spaces}

Let $s\in \N$.  If $c= (c_{\bm,\bn})_{[\bm,\bn]\in \Z^{2d}}$, we introduce the norm
$$\|c\|_{\ell^2_s} := \left(\sum_{[\bm,\bn]\in \Z^{2d}} |c_{\bm,\bn}|^2 (1+ |\bm|^2 + |\bn|^2)^s\right)^{1/2},$$
and denote by $\ell^2_s$ the set of sequences such that this norm is finite.

If $D\in \N$, write $\chi_D$ for the multiplication by the sequence taking value $1$ if $|(\bm,\bn)|\geq D$ and $0$ otherwise.  We then clearly have, for any $s,s'\geq 0$,
$$\|\chi_D\|_{\ell^2_{s+s'} \to \ell^2_s} \leq D^{-s'}.$$

\subsection{Proof of Theorem \ref{theorem_approximation}}

First of all, we note that the function $g(\bx) = e^{-|\bx|^2/2}$ belongs to $M^1_{v_s}$ for any $s$. Therefore, the canonical dual function $\gamma= \Psi^*_{1,0,0}$ does also belong to $M^1_{v_s}$ for all $s$, thanks to \cite[Theorem 13.2.1]{grochenig2001foundations}.

It follows from \cite[Theorem 12.2.4]{grochenig2001foundations} that the restriction to $M_s^2$ of the operator $T_1$ introduced in (\ref{eq:DefCoefOp}) is bounded from $M_s^2$ to $\ell^2_{s}$.  We shall denote this operator by $C_\gamma$.

We also introduce the operator
$D_g : c \mapsto \sum_{m,n\in \Z^{2d}} c_{m,n} \Psi_{m,n}$.
Thanks to \cite[Theorem 12.2.4]{grochenig2001foundations},
$D_g$ is a bounded operator from $\ell^2_s$ to $M^2_s$.  
Furthermore, thanks to \cite[Corollary 12.2.6]{grochenig2001foundations},
for any $u\in M^2_s$, we have $u= D_g C_\gamma u$.

 We therefore have, for any $u\in M_s^2$,
 \begin{align*}
 \left\|u - \sum_{|(\bm,\bn)|\leq \frac{D}{\sqrt{\pi}}} \langle(u, \Psi^*_{1,\bm,\bn}\rangle \Psi_{1,\bm,\bn}\right\|_{M^2_s} &= \left\|D_g \chi_{\frac{D}{\sqrt{\pi}}} C_\gamma u\right\|_{M^2_s}\\
 & \leq C(s,s') \left\|\chi_{\frac{D}{\sqrt{\pi}}}\right\|_{\ell^2_{s+s'}\to \ell^2(s)} \|u\|_{M^2_{s+s'}}\\
 & \leq C(s,s') D^{-s'} \|u\|_ {M^2_{s+s'}},
 \end{align*}
 so that 
$$ \left\|u - \sum_{|(\bm,\bn)|\leq \frac{D}{\sqrt{\pi}}} \langle(u, \Psi^*_{1,\bm,\bn}\rangle \Psi_{1,\bm,\bn}\right\|_{\widetilde{H}_1^s}\leq   C(s,s') D^{-s'} \|u\|_ {\widetilde{H}_1^{s+s'}},$$
so that (\ref{eq:K1}) follows.

\bibliographystyle{amsplain}
\bibliography{bibliography.bib}

\end{document}

%% file: general_commands.tex
\usepackage{amssymb}
\usepackage{mathrsfs}

\renewcommand{\Im}{\operatorname{Im}}

\newcommand{\eq}{:=}

\newcommand{\ba}{\boldsymbol a}

\newcommand{\bb}{\boldsymbol b}

\newcommand{\bm}{\boldsymbol m}
\newcommand{\bn}{\boldsymbol n}

\newcommand{\bp}{\boldsymbol p}
\newcommand{\bq}{\boldsymbol q}

\newcommand{\bv}{\boldsymbol v}

\newcommand{\bx}{\boldsymbol x}
\newcommand{\by}{\boldsymbol y}
\newcommand{\bz}{\boldsymbol z}
\newcommand{\bell}{\boldsymbol \ell}

\newcommand{\CE}{\mathcal E}
\newcommand{\CF}{\mathcal F}

\newcommand{\CS}{\mathcal S}

\newcommand{\N}{\mathbb N}
\newcommand{\Z}{\mathbb Z}
\newcommand{\R}{\mathbb R}
\newcommand{\C}{\mathbb C}

%% file: special_commands.tex
\newcommand{\bxi}{\boldsymbol \xi}

\newcommand{\xxx}[2]{\bx ^{#1,#2}}
\newcommand{\xxi}[2]{\bxi^{#1,#2}}

\newcommand{\gs}[3]{\Psi_{#1,#2,#3}}

\newcommand{\gskmn}{\gs{k}{\bm}{\bn}}
\newcommand{\gskmnp}{\gs{k}{\bm'}{\bn'}}
\newcommand{\gskmnpp}{\gs{k}{\bm''}{\bn''}}

\newcommand{\xxkm}{\xxx{k}{\bm}}
\newcommand{\xikn}{\xxi{k}{\bn}}

\newcommand{\xxkmp}{\xxx{k}{\bm'}}
\newcommand{\xiknp}{\xxi{k}{\bn'}}

\newcommand{\bzero}{\boldsymbol 0}

\newcommand{\dx}{\mathrm{d}\bx}
\newcommand{\dy}{\mathrm{d}\by}